\documentclass[11pt,reqno]{amsart}

\usepackage[english]{babel}
\usepackage{enumitem}
\usepackage[letterpaper,top=2cm,bottom=2cm,left=3cm,right=3cm,marginparwidth=1.75cm]{geometry}

\usepackage{amsmath}
\usepackage{amssymb}
\usepackage{graphicx}
\usepackage{tikz}
\usepackage[colorlinks=true, allcolors=blue]{hyperref}
\theoremstyle{definition}
\newtheorem{theorem}{Theorem}[section]
\newtheorem{corollary}[theorem]{Corollary}
\newtheorem{lemma}[theorem]{Lemma}

\title{Mock modular forms from the $k$-rank moments}
\author{Kilian Rausch}
\address{Department of Mathematics and Computer Science, Division of Mathematics, University of
Cologne, Weyertal 86-90, 50931 Cologne, Germany}
\email{krausch1@uni-koeln.de}

\begin{document}

\subjclass[2020]{11F03, 11F11, 11F37, 11F50, 11P82}
\keywords{completions, divisor like sums,  mock modular forms, partitions, partition Eisenstein traces}
\begin{abstract}
In this paper, the generating functions of Garvans so-called $k$-ranks are used, to define a family of mock Eisenstein series.  The $k$-rank moments are then expressed as partition traces of these functions.  We explore the modular properties of this new family, give recursive formulas for them involving divisor like sums, and prove that their Fourier coefficient are integral.  Furthermore, we show that these functions lie in an algebra that is generated only by derivatives up to a finite order but is nevertheless closed under differentiation. In the process, we also answer a question raised by Bringmann, Pandey and van Ittersum by showing that the divisor like sum 
\begin{align*}
      \left(1-2^{\ell-1} \right) \frac{B_\ell}{2\ell}+ \sum_{2n-1 \geq bm \geq b}  (2n-bm)^{\ell-1} q^{mn} - \sum_{m-1 \geq 2bn \geq 2b} (m-2bn)^{\ell-1} q^{mn},
\end{align*}
has a quasi-completion, when $b\geq 3$ is odd. 
\end{abstract}
%maketitle after abstract so that the keywords are here and there is no weird space between author and abstract
\maketitle
%remove the center part and replace by section
\section{Introduction and statement of results}
A \textit{partition} $\lambda=(\lambda_1,...,\lambda_s)$ of a natural number $n \in \mathbb{N}_0$ is a non increasing sequence of natural numbers such that $\sum_{j=1}^s \lambda_j=n. $ In order to analyze them,  partitions statistics are used. One of these statistics is the \textit{rank of a partition}, which is defined as the difference between the largest part of $\lambda$, denoted by $\ell(\lambda)$, and the number of parts of $\lambda. $ We denote by $N_2(m,n)$ the number ob partitions of $n$ with rank $m$. We deviate from this definition and set $N_2(0,0)=0.$ In the literature $N_2(m,n)$ is usually denoted by $N(m,n)$.   Another partition statistic of interest to us is the \textit{crank}, denoted by  $c(\lambda). $ In order to define the crank \cite{AG1988,fg3}, we first define  $\omega(\lambda)$ as the number of ones in the partition and $\mu(\lambda)$ as the number of parts strictly larger than $\omega(\lambda).$  Then, we have 
\begin{align*}
    c(\lambda):= \begin{cases}
        \ell(\lambda) & \text{\textit{if }} \omega(\lambda)= 0,
        \\ \mu(\lambda)-\omega(\lambda) & \text{\textit{if }} \omega(\lambda)>0. 
    \end{cases}
\end{align*}
Denote by $N_1(m,n)$ the number of partitions of $n$ with crank $m$. This is usually denoted by $M(m,n)$. For $n=1,$ we deviate from this definition and set $N_1(-1,1)=-N_1(0,1)=N_1(1,1):=1.$ As usual, for $\tau \in \mathbb{H}:=\{w \in \mathbb{C}: \text{Im}(w)>0\}$ and $z \in \mathbb{C}$,  here and throughout we set $q:=\exp(2\pi i \tau)$ and $\zeta:=\exp(2\pi i z)$.

In \cite[Theorem 1.2]{amdeberhan2025tracespartitioneisensteinseries} Amdeberhan, Griffin, Ono and Singh and in \cite[Theorem 1.2]{bringmann2025mockeisensteinseriesassociated} Bringmann, Pandey and van Ittersum showed that for $j \in \mathbb{N}_0$ the  \textit{crank moments} 
\begin{align*}
    C_j(q):= \sum_{n=0}^\infty \sum_{m \in \mathbb{Z}} m^j N_1(m,n)q^n \tag{1.1}
\end{align*}
and \textit{rank moments}
\begin{align*}
    R_{2,j}(q):= 1+\sum_{n=0}^\infty \sum_{m \in \mathbb{Z}} m^j N_2(m,n)q^n,\tag{1.2}
\end{align*}
respectively, can be expressed in terms of partition traces, which we will define below.  For the crank moments, partition traces of Eisenstein series suffice, while the rank moments can be expressed in terms of traces of a family of mock Eisenstein series. 
Let $f=\{f_k\}_{k\in \mathbb{N}}$ be a family of functions, $\lambda=(1^{\ell_1},2^{\ell_2},...,n^{\ell_n}) \vdash n$ and $\phi$ a function on partitions. Then we set
\begin{align*}
    f_\lambda(\tau):= \prod_{k=1}^n f_k^{\ell_k}(\tau)
\end{align*}
and define the \textit{$n$-th partition traces} with respect to the family $f$ and the function $\phi$ by
\begin{align*}
    \text{Tr}_n(\phi,f;\tau):= \sum_{\lambda \vdash n} \phi(\lambda)f_\lambda(\tau). 
\end{align*}
For our purposes, we set
\begin{align*}
    \phi(\lambda):= \prod_{k=1}^n \frac{2^{\ell_k}}{\ell_k!k!^{\ell_k}}.  \tag{1.3}
\end{align*} 
We denote by $G=\{G_k\}_{k\in \mathbb{N}}$ the family of Eisenstein series, that is for $\tau \in \mathbb{H}$  and even $k \geq 2$,
\begin{align*}
    G_k(\tau):= -\frac{B_k}{2k } + \sum_{n=1}^\infty \sigma_{k-1}(n)q^n
\end{align*}
and $G_k(\tau)\equiv 0$ for odd $k.$ Furthermore, let $B_k $ denote the $k$-th Bernoulli number. 
In \cite[Theorem 1.2]{amdeberhan2025tracespartitioneisensteinseries} the authors proved the following theorem:
\begin{theorem}
\textit{It holds that}
\begin{align*}
    \sum_{j=0}^\infty C_j(q) \frac{z^j}{j!} = \frac{2\sinh\left(\frac{z}{2} \right)}{z(q)_\infty} \sum_{j=0}^\infty \text{Tr}_j(\phi,G;\tau) z^j.
\end{align*}  
\end{theorem}
As usual, we set $(a)_n:= \prod_{k=0}^{n-1} (1-aq^k)$ for $n \in \mathbb{N}_0 \cup \{\infty\}$.  The natural question, whether a similar result to this exists for the rank moments was confirmed by the authors of \cite{bringmann2025mockeisensteinseriesassociated}. They proved the following theorem (see Theorems 1.2, 1.4 and 1.5 in \cite{bringmann2025mockeisensteinseriesassociated}):
\begin{theorem}\label{thm:bring}
   \textit{There exists a family $h=\{h_j\}_{j\in \mathbb{N}}$ such that}
\begin{align*}
    \sum_{j=0}^\infty R_{2,j}(q)\frac{z^j}{j!} = \frac{2\sinh\left(\frac{z}{2} \right)}{z(q)_\infty} \sum_{j=0}^\infty \text{Tr}_j(\phi,h;\tau) z^j.
\end{align*}
\textit{The $h_j$ satisfy the following properties:}
\begin{itemize}[leftmargin=0.25in]
    \item[$(1)$]{\textit{We have $h_j \equiv 0$ if $j$ is odd and for $j \geq 2$  it holds that} \begin{align*}
        \lim_{\tau \to i\infty} h_j(\tau)= -\frac{B_j}{2j}.
    \end{align*}}
    \item[$(2)$]{\textit{The function $h_j$ has a quasi-completion\footnote{See Subsection 2.2 for the definition.} $h_j^*$  which satisfies}
    \begin{align*}
        h_j^*\left(  \frac{a\tau+b}{c\tau+d}\right) = \begin{cases}
            (c\tau+d)^j h_j^*(\tau) & \text{ \textit{if} } j \neq 2, \\
            (c\tau+d)^2 h_2^*(\tau) + \frac{3ic}{4\pi } (c\tau+d) & \text{ \textit{if} } j=2 
        \end{cases}
    \end{align*}
    \textit{for all} $\left( {a \atop c} {b \atop d} \right) \in \text{SL}_2(\mathbb{Z}).$}
    \item[$(3)$]{\textit{The algebra $\mathcal{G}:=\mathbb{Q}[h_2,h_4,...,G_2,G_4,...]$ is closed under the action of $D:=\frac{1}{2\pi i} \frac{\partial}{\partial \tau}.$ }}
    \item[$(4)$]{\textit{The $h_j$ are uniquely determined via recursive formulas containing divisor-like sums. }}
    \item[$(5)$]{\textit{The Fourier coefficients of $h_j+ \frac{B_j}{2j}$ are integers. }}
    \end{itemize}
\end{theorem} 
Note, that properties (1)-(3) are preserved when adding any weight $j$ cusp form to $h_j,$ hence they do not uniquely determine the $h_j$ in question and hence why we require (4) in Theorem \ref{thm:bring}.  

A \textit{quasi-completion}\footnote{Here and throughout we consider the variables $\tau$ and $\overline{\tau}$ as independent.} $g^*(\tau,\overline{\tau})$  of a function $g(\tau)$ is a real-analytic function, that transforms like a quasimodular form and satisfies $\lim_{\overline{\tau} \to -i \infty} g^*(\tau,\overline{\tau})=g(\tau).$ We call it a \textit{completion}, if it transforms like a modular form. 

In view of Theorem \ref{thm:bring}, it is natural to ask, whether a similar result also holds for the moments of a family of statistic called the $k$-ranks, to which the rank belongs. As our first result (see Theorem \ref{thm:main}), we show that this is indeed the case. In \cite{fg3} Garvan introduced these  $k$\textit{-ranks}, denoted by $r_k$, and gave a combinatorial interpretation for them. For a partition $\lambda,$ we denote by $d_1(\lambda),d_2(\lambda),...$ the sizes of the successive Durfee squares (see Subsection 2.1) contained in its Ferrers diagram. For $k \geq 2,$ we define 
\begin{align*}
    r_k(\lambda) &:= \text{number of columns in the Ferrers graph of $\lambda$ to the right} \\ &\hphantom{:=i}\text{of the first Durfee square, whose length is $\leq d_{k-1}(\lambda)$}   
    \\ &\hphantom{=i} - \text{number of parts of $\lambda$ below the $(k-1)$th Durfee square.}
\end{align*}
If the Ferrers diagram of $\lambda$ contains less than $k-1$ successive Durfee squares, then $r_k(\lambda)=0.$ Note, that the 2-rank is the usual rank. For $k\geq 2$, we denote by $N_k(m,n)$ the number of partitions of $n$ with at least $k-1$ successive Durfee squares and $k$-rank $m.$ By convention we set $N_k(m,0)=0$ for all $m \in \mathbb{Z}.$ 

Our central object of study in this paper are the \textit{$k$-rank moments} for $k \geq 3$: For $j \in \mathbb{N}_0$ we set 
\begin{align*}
    R_{k,j}(q):= \sum_{n =0}^\infty \sum_{m \in \mathbb{Z}} m^j N_k(m,n)q^n.  \tag{1.4}
\end{align*}
In order to state Theorem \ref{thm:main}, we define for $\tau \in \mathbb{H}, m \in \mathbb{N}_0, b \in \mathbb{N}$ and $a\in\mathbb{Z}$ the $q$-series,
    \begin{align*}
        \theta_{a,b}(\tau):= \sum_{n \in \mathbb{Z}} (-1)^n q^{\frac{bn^2+an}{2}} \tag{1.5}
    \end{align*}
    and
    \begin{align*}
        \theta_{a,b}^{[m]}(\tau):= \left( \frac{1}{\pi i} \frac{\partial}{ \partial \tau} \right)^m  \theta_{a,b}(\tau)= \sum_{n\in\mathbb{Z}} (-1)^n \left(bn^2+an \right)^m q^{\frac{(bn^2+an)}{2}}. \tag{1.6}
    \end{align*}
\begin{theorem}\label{thm:main}
\textit{For any integer $k \geq 3,$ there exists a family $f_k=\{f_{k,j}\}_{j\in \mathbb{N}}$ such that}
\begin{align*}
    \sum_{j=0}^\infty R_{k,j}(q) \frac{z^j}{j!}= \frac{2 \sinh \left(\frac{z}{2}\right)}{z(q)_\infty} \sum_{j=0}^\infty \text{Tr}_j(\phi,f_k;\tau) z^j. 
\end{align*}
\textit{The $f_{k,j}$ satisfy the following properties: }
\begin{itemize}[leftmargin=0.25in]
    \item[(1)]{\textit{We have $f_{k,j} \equiv0$ if $j$ is odd and for $j\geq 2$ it holds that\begin{align*}
        \lim_{\tau \to i\infty} f_{k,j}(\tau)= -\frac{B_j}{2j}. 
    \end{align*} }}
    \item[(2)]{\textit{The function $f_{k,j}$ has a quasi-completion $f_{k,j}^*$ which satisfies}
    \begin{align*}
        f_{k,j}^*\left(  \frac{a\tau+b}{c\tau+d}\right) = \begin{cases}
            (c\tau+d)^j f_{k,j}^*(\tau) & \textit{if } j \neq 2, \\
            (c\tau+d)^2 f_{k,2}^*(\tau) + \frac{(2k-1)ic}{4\pi } (c\tau+d) & \textit{if } j=2 
        \end{cases}
    \end{align*}
    \textit{for all} $\left( {a \atop c} {b \atop d} \right) \in \text{SL}_2(\mathbb{Z}).$}
    \item[(3)]{\textit{For $k=3$ the algebra $\mathcal{F}:=\mathbb{C}[f_{3,2},f_{3,2}',f_{3,4},f_{3,4}',...,G_2,G_4,...,\theta_{1,5},\theta_{1,5}^{[1]},\theta_{3,5},\theta_{3,5}^{[1]}]$ is closed under the action of $D.$   }  }      
\end{itemize}
\end{theorem} 
Similar to the $h_j$ in \cite{bringmann2025mockeisensteinseriesassociated} these two (or in the case $k=3$, three) properties do not uniquely determine the $f_{k,j},$ because the function resulting from adding a cusp form of weight $j$ to $f_{k,j}$ still satisfies all these properties. Uniqueness of the $h_j$ in \cite{bringmann2025mockeisensteinseriesassociated} was achieved by using recursive formulas involving divisor like sums, a result that also holds for the $f_{k,j}$ considered in this paper: For integers $a,b$ and $\ell \in  \mathbb{N}_0,$  we define $g_{a,b,0}:=1,$ $g_{a,b,\ell}:=0$ for odd $\ell$ and
\begin{align*}
     g_{a,b,\ell}(\tau):= \left(1-2^{\ell-1} \right) \frac{B_\ell}{2\ell}+ \sum_{an-1 \geq bm \geq b}  (an-bm)^{\ell-1} q^{mn} - \sum_{m-1 \geq abn \geq ab} (m-abn)^{\ell-1} q^{mn},
\end{align*}
for even $\ell.$  Bringmann, Pandey and van Ittersum used the family $g_{2,3,\ell}$  to uniquely determine the $h_j$ in Theorem \ref{thm:bring} (4). They also raised the question, whether there are other combinations of $a$ and $b$ such that the $g_{a,b,\ell}$ has a quasi-completion. We will show that this is indeed the case, if $a=2$ and $b=2k-1$ is any odd number $\geq 3.$ 
\begin{theorem}\label{thm:rec}
\textit{Let $f_k=\{f_{k,j}\}_{j \in \mathbb{N}}$ be the family of functions from Theorem \ref{thm:main}  and $n \in \mathbb{N}.$ Then we have
\begin{align*}
f_{k,n}(\tau) &=  \frac{ng_{2,2k-1,n}(\tau)}{2^{n-1}} - \sum_{\ell=2}^{n-1 } \binom{n-1}{\ell-1}f_{k,\ell}(\tau)  \frac{(n-\ell)g_{2,2k-1,n-\ell}(\tau)}{2^{n-\ell-2}}
 \end{align*}
and}
\begin{align*}
 f_{k,n}(\tau) = \sum_{\ell=2}^{n}  \frac{\ell g_{2,2k-1,\ell}(\tau)}{2^{\ell-1}} \frac{(n-1)!}{(\ell-1)!}\text{Tr}_{n-\ell}(\psi,f_k;\tau), 
\end{align*}
\textit{where $\psi(\lambda):=(-1)^{\sum_{k=1}^n \ell_k} \phi(\lambda)$. }
\end{theorem}
From Theorems \ref{thm:bring}, \ref{thm:main} and \ref{thm:rec} we get the following corollary.
\begin{corollary}
  \textit{For odd $b\geq 3$ the $g_{2,b,l}(\tau) $ has a quasi-completion. }
\end{corollary}
Our last result concerns the integrality of the Fourier coefficients of the $f_{k,j}$.
\begin{theorem}\label{thm:integral}
\textit{The Fourier coefficients of $f_{k,j}(\tau) + \frac{B_j}{2j}$ are integers.}
\end{theorem}
Note, that our results for the $f_{k,j}$ from Theorem \ref{thm:main} (i) and (ii), Theorem \ref{thm:rec} and Theorem \ref{thm:integral} are similar to the results from \cite{bringmann2025mockeisensteinseriesassociated} for their $h_j$, however the key difference in our approach lies in the definition of the $f_{k,j}.$ These are not only defined in terms of the generating function of the $k$-rank but as a sum of  the generating function and $\tfrac{\theta_{1,2k-1}(\tau)}{(q)_\infty}$ (see (3.1)). This additional term arises in our calculations in Subsection 2.5,  had to be accounted for in all subsequent calculations and is necessary in order to guarantee that the $f_{k,j}^*$ in Theorem \ref{thm:main} (2) are quasi-completions of the $f_{k,j}.$ 

The authors of \cite{bringmann2025mockeisensteinseriesassociated} could make use of the fact that the generating functions of the crank and the rank are linked via a partial differential equation.  For the general $k$-ranks we had to resort to rank-crank type PDEs involving \textit{level l Appell series} (see Subsections 2.4 and 2.7). These higher order differential equations with considerably more terms than the one in \cite{bringmann2025mockeisensteinseriesassociated} 
made the calculations way more intricate and resulted in a larger amount of generators with higher order of derivatives needed in order to define the algebra in (3) of Theorem \ref{thm:main} (and those in Theorem \ref{thm:alge}). In order to ensure that only finite order  derivatives of these generators were needed, we had to utilize an additional differential equation. 

This paper is structured as follows:  In Section 2 we supply the necessary preliminaries used in the later sections. In the third section proofs for (1) and (2) of Theorem \ref{thm:main} as well as the representation of the $R_{k,j}(q)$ in term of traces, are provided. (3) of Theorem \ref{thm:main} is proved in Section 4. Here we also sketch a proof of a more general result, regarding the functions $f_{k,j}$. In the fifth section proofs for Theorems \ref{thm:rec} and \ref{thm:integral} are given. Afterwards we also give some examples of the Fourier coefficients of some of these functions. 

 \section*{Acknowledgments}

I would like to thank my dissertation advisor Kathrin Bringmann for suggesting the topic of this paper and guidance in the  research endeavor. I also want to thank my colleagues Johann Franke, Caner Nazaroglu,  Badri  Vishal Pandey, Johann Stumpenhusen and Jan-Willem van Ittersum  for helpful comments on previous versions of this paper and insightful discussions. This research was funded by the European Research Council (ERC) under the European Union´s Horizon 2020 research and innovation programme (grant agreement No. 101001179). 

\section{Preliminaries}
\subsection{Successive Durfee squares and $k$-ranks} The \textit{Durfee square} of a partition $\lambda$, visually speaking, is the largest square contained in its Ferres diagram. The size of this square is the largest natural number $r$ such that $\lambda$ contains at least $r$ parts, all at least of size $r$.  The square is named after the American mathematician William Pitt Durfee \cite{andrews}.  

In order to calculate the sizes of successive Durfee squares, we proceed as follows: $d_1(\lambda)$ is simply the size of the Durfee square of $\lambda$. This square splits the graphical representation of $\lambda$ in three sections: the square itself, the region to the right of the square and the region below the square. If the resulting region below the square is non empty, $d_2(\lambda)$ is the size of its Durfee square. To determine the other $d_\ell(\lambda)$ continue this process, until the region below the last Durfee square is empty. If the region below the $n$-th successive Durfee square is empty, we have $d_m(\lambda)=0$ for all $m \geq n+1.$ This process is also explained in \cite[Section 3]{fg3}.

To illustrate, we give the following example: Here $\lambda=(7,4,4,3,2,1) \vdash 21: $
\\\begin{center}
\begin{tikzpicture}[scale=0.7] 
\node at (1,1) [circle,fill,inner sep=1.5pt]{};
\node at (1,2) [circle,fill,inner sep=1.5pt]{};
\node at (1,3) [circle,fill,inner sep=1.5pt]{};
\node at (1,4) [circle,fill,inner sep=1.5pt]{};
\node at (1,5) [circle,fill,inner sep=1.5pt]{};
\node at (1,6) [circle,fill,inner sep=1.5pt]{};
\node at (2,2) [circle,fill,inner sep=1.5pt]{};
\node at (2,3) [circle,fill,inner sep=1.5pt]{};
\node at (2,4) [circle,fill,inner sep=1.5pt]{};
\node at (2,5) [circle,fill,inner sep=1.5pt]{};
\node at (2,6) [circle,fill,inner sep=1.5pt]{};
\node at (3,3) [circle,fill,inner sep=1.5pt]{};
\node at (3,4) [circle,fill,inner sep=1.5pt]{};
\node at (3,5) [circle,fill,inner sep=1.5pt]{};
\node at (3,6) [circle,fill,inner sep=1.5pt]{};
\node at (4,4) [circle,fill,inner sep=1.5pt]{};
\node at (4,5) [circle,fill,inner sep=1.5pt]{};
\node at (4,6) [circle,fill,inner sep=1.5pt]{};
\node at (5,6) [circle,fill,inner sep=1.5pt]{};
\node at (6,6) [circle,fill,inner sep=1.5pt]{};
\node at (7,6) [circle,fill,inner sep=1.5pt]{};
\draw (0.5,6.5) -- (3.5,6.5);
\draw (3.5,6.5) -- (3.5,3.5);
\draw (0.5,3.5) -- (0.5,6.5);
\draw (0.5,3.5) -- (3.5,3.5);
\draw (0.5,1.5) -- (0.5,3.5);
\draw (0.5,1.5) -- (2.5,1.5);
\draw (2.5,1.5) -- (2.5,3.5);
\draw (0.5,0.5) -- (0.5,1.5);
\draw (0.5,0.5) -- (1.5,0.5);
\draw (1.5,0.5) -- (1.5,1.5);
\end{tikzpicture}
\end{center}
\begin{center}
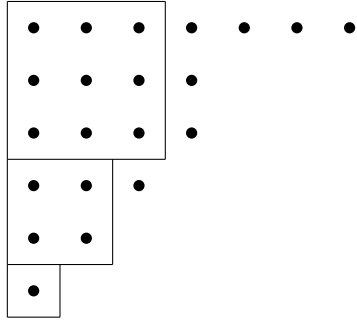
    \textit{Fig. 1: Successive Durfee squares of $\lambda=(7,4,4,3,2,1)$} 
\end{center}
The first Durfee square has a size of $3.$ The second one right below it has a size of $2$, while the third has size 1. Since there are no further nodes below the third square, we have $d_1(\lambda)=3, d_2(\lambda)=2, d_3(\lambda)=1$ and $ d_k(\lambda)=0 $  for $k \geq 4.$ 

Garvan \cite[Section 4]{fg3}  showed that for $k \geq 3$ the generating function of $N_k(m,n)$ is given by
\begin{align*}
     FG_k(\zeta,q)&:=\sum_{n=0}^\infty \sum_{m\in \mathbb{Z}} N_k(m,n)\zeta^mq^n
     \\ &\hphantom{:}= \sum_{n_{k-1} \geq n_{k-2}\geq ... \geq n_2 \geq n_1 \geq 1} \frac{q^{n_1^2+...+n_{k-1}^2}}{(q)_{n_{k-1}-n_{k-2}}...(q)_{n_2-n_1}(\zeta q)_{n_1}(\zeta^{-1}q)_{n_1}}, 
 \end{align*}
 for $|q|<|\zeta|<|q|^{-1}.$
\subsection{Quasimodular forms and Eisenstein series.} For integral $k \geq 3$ the Eisenstein series $G_k$ are the canonical examples of modular forms of  weight $k$. That is, for $\left ( {a \atop c} {b \atop d} \right)\in \text{SL}_2(\mathbb{Z}),$ we have
\begin{align*}
    G_k \left(\frac{a\tau+b}{c\tau+d} \right) = (c\tau+d)^k G_k(\tau). 
\end{align*}
The function $G_2$ itself does not constitute a modular form of weight $2$, because we have
\begin{align*}
    G_2 \left(\frac{a\tau+b}{c\tau+d} \right) = (c\tau+d)^2 G_2(\tau) + \frac{ic}{4\pi} (c\tau+d). \tag{2.1}
\end{align*}
However, by adding a non-holomorphic term $G_2$ can be made modular. More precisely, for $\tau=:u+iv, u,v \in \mathbb{R}$ we set 
\begin{align*}
    G_2^*(\tau):= G_2(\tau) + \frac{1}{8 \pi v}.
\end{align*}
Then we obtain the transformation
\begin{align*}
    G_2^*\left(\frac{a\tau+b}{c\tau+d} \right) = (c\tau+d)^2 G_2^*(\tau).
\end{align*}
For a more in depth discussion of the Eisenstein series see for example \cite{KK} or \cite{don}.
Because $G_2$ can be recovered from $G_2^*$ by taking the limit $\overline{\tau} \to -i\infty,$ it constitutes a quasimodular form (we will define that in a moment). By their definition, the Eisenstein series are closely related to the Bernoulli numbers, we have
\begin{align*}
    \lim_{\tau \to i\infty} G_k(\tau)= - \frac{B_k}{2k}. 
\end{align*}
Another well-known property of the Eisenstein series is that the algebra $\mathbb{C}[G_2,G_4,G_6,...]$ of quasimodular forms is closed under the action of $ \frac{1}{2\pi i}\frac{\partial}{\partial \tau}.$ 

We call a function $F: \mathbb{H} \to \mathbb{C}$ an \textit{almost holomorphic modular form} of weight $k$ and depth $s$, if the following conditions hold:
\begin{itemize}[leftmargin=0.25in]
    \item[(1)] for all $\left( {a \atop c}{b \atop d}\right) \in \text{SL}_2(\mathbb{Z})$ it holds that
    \begin{align*}
        F\left( \frac{a\tau+b}{c\tau+d} \right) = (c\tau+d)^k F(\tau),
    \end{align*}
    \item [(2)] there exist finitely many holomorphic functions $f_j : \mathbb{H} \to \mathbb{C}$ with $f_s \not\equiv  0$ and
    \begin{align*}
        F(\tau):= \sum_{j=0}^s \frac{f_j(\tau)}{v^j},
    \end{align*}
    \item[(3)] and $F$ grows at most polynomially in $v^{-1} $ as $v \to 0.$
\end{itemize}
By convention, we set the depth of the zero function as $-\infty.$ We call $f_0$ a \textit{quasimodular form} of weight $k$ and depth $s.$ It can be obtained from $F$ by
\begin{align*}
    \lim_{\overline{\tau} \to -i \infty} F(\tau)=f_0(\tau).
\end{align*}
Further, an analytical function $G(\tau,\overline{\tau})$ is said to transform like a quasimodular form if there exist real-analytic functions $g_j(\tau,\overline{\tau})$ so that
\begin{align*}
    (c\tau+d)^{-k} G\left(\frac{a\tau+b}{c\tau+d}, \frac{a\overline{\tau}+b}{c\overline{\tau}+d} \right) = \sum_{j=0}^s g_j(\tau,\overline{\tau}) \left( \frac{c}{c\tau+d} \right)^j.
\end{align*}
\subsection{Crank moments.} In \cite[Theorem 1.2]{amdeberhan2024derivativesthetafunctionstraces} the authors showed that the generating function of the crank  
\begin{align*}
    C(\zeta,q):= \sum_{n=0}^\infty  \sum_{m\in\mathbb{Z}} N_1(m,n)\zeta^mq^n ,
\end{align*}
can be rewritten as 
\begin{align*}
    C(\zeta,q)= \frac{\sin(\pi z)}{\pi z (q)_\infty} \exp\left(2 \sum_{k= 2}^\infty G_k(\tau) \frac{(2\pi i z)^k}{k!} \right). \tag{2.2}
\end{align*} 
They concluded from this (see Lemma 3.1) the following lemma:
\begin{lemma}\label{lem:bern}
\textit{It holds that}
\begin{align*}
    \frac{\zeta^{\frac{1}{2}}}{\zeta-1} = \frac{1}{2\pi i z} \exp \left( - \sum_{k= 2}^\infty \frac{B_k}{k} \frac{(2\pi i z)^k}{k!} \right). \tag{2.3}
\end{align*}
\end{lemma} 
The Bernoulli numbers $B_k$ are related to the Bernoulli polynomials $B_k(x)$, as they are the constant term of the corresponding polynomial. The Bernoulli polynomials satisfy 
\begin{align*}
    \frac{ze^{zx}}{e^z-1}= \sum_{k=0}^\infty  B_k(x) \frac{z^k}{k!}. \tag{2.4}
\end{align*}A useful property of the Bernoulli polynomials is that
\begin{align*}
    B_k(x+y)= \sum_{n=0}^k  \binom{k}{n} B_{k-n}(x)y^n. \tag{2.5}
\end{align*}
\subsection{Mock modularity of the $k$-rank generating function.} Garvan \cite[Theorem 1.12]{fg3} showed that the number of partitions of $n$ with $k$-rank $m$, $N_k(m,n)$,  satisfy\footnote{To be more precise: Garvan defined $N_k(m,n)$ as the coefficients of the $q$-series on the right-hand side. He then definition the $k$-rank to explain $N_k$ combinatorially.}
\begin{align*}
    \sum_{n=0}^\infty N_k(m,n)q^n  = \frac{1}{(q)_\infty} \sum_{n=1}^\infty (-1)^{n-1} q^{\frac{n((2k-1)n-1)}{2} +|m|n}(1-q^n) \tag{2.6}
\end{align*} 
Using this, one calculates
\begin{align*}
    FG_k(\zeta,q) %&= \sum_{n =0}^\infty \sum_{m \in \mathbb{Z}}  N_k(m,n) \zeta^m q^n  \\
      %&= \sum_{m \in \mathbb{Z}} \zeta^m \frac{1}{(q)_\infty} \sum_{n=1}^\infty (-1)^{n-1} q^{((2k-1)n^2-n)/2+|m|n}(1-q^n)
      &= \frac{1}{(q)_\infty} \sum_{n =1}^\infty (-1)^{n+1}\frac{q^{ \frac{(2k-1)n^2-n}{2}}(1-q^n)(1-q^{2n})}{(1-\zeta q^n)(1-\zeta^{-1}q^n)}.
\end{align*} 
In \cite[Theorem 1.1]{Dixit} Chan, Dixit and Garvan showed that the $k$-rank generating functions are closely related to the \textit{level $\ell$ Appell series} for  $\ell \in \mathbb{N}$:
\begin{align*}
    A_\ell(z,w,\tau):= \zeta^{\frac{\ell}{2}} \sum_{n \in \mathbb{Z}} (-1)^{\ell n} \frac{q^{\frac{\ell n(n+1)}{2}}e^{2\pi in w}}{1-\zeta q^n}, 
\end{align*}
 studied by Zwegers in \cite{zwegersappell}. They showed that
\begin{align*}
    FG_k(\zeta,q) &= \frac{1}{(q)_\infty} \bigg[ \left(\zeta^{-\frac{1}{2}} - \zeta^{\frac{1}{2}} \right) A_{2k-1}(z,0,\tau) - \zeta \theta_{1,2k-1}(\tau) \\&\hspace{5cm}+ \zeta(1-\zeta) \sum_{m=0}^{k-3} \zeta^m \theta_{2m+3,2k-1}(\tau)\bigg], 
 \end{align*}
 Utilizing $\zeta^{-\frac{1}{2}}-\zeta^{\frac{1}{2}}=-2i\sin(\pi z),$ this can be rearranged to 
 \begin{align*}
     A_{2k-1}(z,0,\tau)= \frac{(q)_\infty FG_k(\zeta,q)}{-2i\sin( \pi z)} - \frac{\zeta}{2i \sin(\pi z)} \theta_{1,2k-1}(\tau) - \zeta^{\frac{3}{2}} \sum_{m=0}^{k-3} \zeta^m \theta_{2m+3,2k-1}(\tau).
 \end{align*}
 
 In \cite[Theorem 3]{zwegersappell} Zwegers showed that the Appell function $A_\ell$ can be completed to a Jacobi function $\hat{A}_\ell$, which is defined as:
 \begin{align*}
     \widehat{A}_{2k-1}(z,0,\tau):= A_{2k-1}(z,0,\tau)+H_k(z,\tau), 
 \end{align*}
 where we have
 \begin{align*}
     H_k(z,\tau):= \frac{i}{2} \sum_{\ell=0}^{2k-2} \zeta^{\ell} \theta(\ell\tau,(2k-1)\tau) R((2k-1)z-\ell\tau,(2k-1)\tau), \tag{2.7}
 \end{align*}
 with
\begin{align*}
    \theta(z,\tau):= \sum_{v \in \mathbb{Z} +\frac{1}{2}} e^{\pi i v^2\tau+2\pi i v \left(z+ \frac{1}{2} \right)}, \tag{2.8}
\end{align*}
\begin{align*}
    R(z,\tau):= \sum_{v \in \mathbb{Z} + \frac{1}{2}}\left[\text{sgn}(v)-E \left(\left(v+\frac{\text{Im}(z)}{\text{Im}(\tau)}\right)\sqrt{2\text{Im}(\tau)}\right) \right](-1)^{v- \frac{1}{2}}q^{-\frac{v^2}{2}}e^{-2\pi i v z} . \tag{2.9}
\end{align*}
Here, we have
\begin{align*}
    E(y):= 2 \int_0^y e^{-\pi t^2}dt=\text{sgn}(y) (1-\beta(y^2)), 
\end{align*}
with
\begin{align*}
    \beta(y):=\int_y^\infty u^{-\frac{1}{2}} e^{-\pi u} du. 
\end{align*} As we will be only interested in Appell series where the second argument $u=0,$ we denote them by 
\begin{align*}
    A_\ell(z,\tau):=A_\ell(z,0,\tau).
\end{align*}
 From  \cite[Theorem 4]{zwegersappell}, we get that for $\left( {a \atop c} {b \atop d}  \right) \in \text{SL}_2(\mathbb{Z})$ it holds that
 \begin{align*}
     \widehat{A}_{2k-1}  \left(\frac{z}{c\tau+d}, \frac{a\tau+b}{c\tau+d} \right)= (c\tau +d) e^{\frac{-(2k-1) \pi i c z^2}{(c\tau+d)}}\widehat{A}_{2k-1}(z,\tau). \tag{2.10}
 \end{align*}
 In light of this, we define
   \begin{align*}
     \widehat{FG}_k (z,\overline{z},\tau,\overline{\tau})&:= FG_k(\zeta,q)- \frac{2i\sin(\pi z)}{(q)_\infty}H_k(z,\tau)  \\&\;\hspace{1cm}+ \frac{\zeta}{(q)_\infty}  \theta_{1,2k-1}(\tau) - \frac{\zeta(1-\zeta)}{(q)_\infty} \sum_{m=0}^{k-3} \zeta^m \theta_{2m+3,2k-1}(\tau)
     \\ &\;= \frac{-2i\sin(\pi z)}{(q)_\infty} \widehat{A}_{2k-1}(z,\tau).
 \end{align*}
 Rewriting this and factoring in another normalizing term, we get
 \begin{align*}
     \eta(\tau)  K_k(z,\tau) = \widehat{A}_{2k-1}(z,\tau) e^{4(2k-1)\pi^2z^2G_2(\tau)},
 \end{align*}
 with
 \begin{align*}
    K_k(z,\tau):=  \frac{\widehat{FG}_k(z,\overline{z},\tau,\overline{\tau})q^{-\frac{1}{24}}i}{2\sin(\pi z)} e^{4(2k-1)\pi^2 z^2 G_2(\tau)}
 \end{align*}
 Here $\eta(\tau):= q^{\frac{1}{24}} (q)_\infty $ is the usual Dedekind eta-function. Based on (2.10), we have the following  lemma:
\begin{lemma}\label{lem:trans}
\textit{For}  $\left( {a \atop c} {b \atop d}  \right) \in \text{SL}_2(\mathbb{Z})$ \textit{it holds that
 \begin{align*}
     \eta\left(\frac{a\tau+d}{c\tau+d} \right)  K_k\left( \frac{z}{c\tau+d},\frac{a\tau+b}{c\tau+d}\right) = (c\tau+d) \eta(\tau)  K_k(z,\tau), \tag{2.11}
 \end{align*}
 that is, $\eta(\tau)\cdot K_k(z,\tau)$ modular transforms like a Jacobi form of weight $1$ and index $0.$ }
 \end{lemma} 
\subsection{The holomorphic part of $H_k$ and $\widehat{A}_{2k-1}$ in $\tau$.} For our proceedings, we need to better understand the function $H_k(z,\tau),$ especially when taking the limit $\overline{\tau} \to -i\infty.$ For this we need some preliminary results. We start with a result giving a relation between the two  $\theta$-functions introduced in (1.5) and (2.8). Using these and some simple algebraic manipulations, one verifies the following lemma: 
\begin{lemma}\label{lem:theta} 
\textit{For $\ell=k,...,2k-2$ we have,}
\begin{align*}
   i q^{-\frac{2k-1}{8}+ \frac{\ell}{2}} \theta(\ell\tau,(2k-1)\tau)= \theta_{2\ell-(2k-1),2k-1}(\tau).
\end{align*}
\end{lemma}
%\textbf{Proof:} By (1.11), we have
%\begin{align*}
% &\hphantom{=}i q^{-\frac{2w-1}{8}+ \frac{k}{2}} \theta(k\tau,(2w-1)\tau)
% \\ &= i\sum_{n \in \mathbb{Z}} e^{(2w-1)\pi i \tau \left(n+\frac{1}{2} \right)^2 + 2\pi i\left(  n + \frac{1}{2}\right) \left(k\tau+ \frac{1}{2} \right)} e^{2\pi i \tau \left(- \frac{2w-1}{8}+\frac{k}{2} \right)}
%   \\&= i \sum_{n \in \mathbb{Z}} (-1)^n i q^{\frac{2w-1}{2}\left(n^2+n+\frac{1}{4} \right) +  k n  +\frac{k}{2}  } q^{- \frac{2w-1}{8}+\frac{k}{2} } 
%   \\ &= \sum_{n \in \mathbb{Z}} (-1)^{n+1} q^{(2w-1)(n^2+2n+1)/2 + ( 2k - (2w-1)) (n+1)/2  }
%   \\ &= \sum_{n \in \mathbb{Z}} (-1)^{n+1} q^{((2w-1)(n+1)^2 + (2 k - (2w-1)) (n+1))/2}
%   \\ &= \sum_{n \in \mathbb{Z}}(-1)^{n} q^{((2w-1)n^2 + (2 k - (2w-1)) n)/2}
%   \\&= \theta_{2k-(2w-1),2w-1}(\tau).
%\end{align*} This gives the claim. $\hfill \square $ \\
Next, we want to study the limiting behavior of the $R$-terms in the definition of $H_k$. This result is stated in the following lemma: 
\begin{lemma}\label{lem:limitR}
\textit{We have
\begin{align*}
    \lim_{\overline{\tau} \to -i\infty } R((2k-1)z-\ell\tau,(2k-1)\tau)= \begin{cases}
        0 & \text{if } \ell=0,...,k-1,\\
        2 q^{-\frac{2k-1}{8} + \frac{\ell}{2}} \zeta^{- \frac{2k-1}{2}} &  \text{if } \ell=k,...,2k-2.
    \end{cases}
\end{align*}}
\end{lemma}
\begin{proof} Using $d:=2k-1$ and (2.9), we get %schon angepasst an neue nummerierung
\begin{align*}
    &\lim_{\overline{\tau} \rightarrow-i\infty} R(dz -\ell \tau,d\tau)
    %&= \sum_{\ell \in \mathbb{Z}+1/2} \left(\text{sgn}(\ell)-\lim_{\overline{\tau} \rightarrow -i\infty} E\left[\left(\ell+\frac{\text{Im}( dz-k\tau)}{d\text{Im}(\tau)} \right) \sqrt{2d\text{Im}(\tau)}\right]\right) (-1)^{\ell-1/2}q^{-d\ell^2/2} e^{-2\pi i\ell(dz-k\tau)}
    \\ &\hspace{1cm}= \sum_{j \in \mathbb{Z}+ \frac{1}{2}} \left(\text{sgn}(j)-\lim_{ v\rightarrow \infty} E\left[ \sqrt{2d}\left(\left(j-\frac{\ell}{d} \right)\sqrt{v}+\frac{ y}{ \sqrt{v}} \right) \right]\right) (-1)^{j-\frac{1}{2}}q^{- \frac{d j^2}{2}+\ell j} \zeta^{-dj}.
\end{align*}
For $\ell=0,...,k-1$ the terms on the right-hand side vanishes, because we have
\begin{align*}
     \lim_{v \to \infty} E\left(\sqrt{2dv} \left(n+\frac{1}{2}-\frac{\ell}{d} + \frac{y}{v}\right) \right) &=\text{sgn}\left(n+ \frac{1}{2}- \frac{\ell}{d} \right) 
     = \begin{cases} 
     1 & \text{\textit{if} } n > \frac{\ell}{d}- \frac{1}{2},\\ 
     0 & \text{\textit{if} } n= \frac{\ell}{d}- \frac{1}{2},\\ 
     -1 & \text{\textit{if} } n< \frac{\ell}{d}- \frac{1}{2}        
    \end{cases} 
    \\&=\text{sgn}\left(n+\frac{1}{2}\right)
\end{align*}
for all $n \in \mathbb{Z}.$ For $\ell=k,...,2k-2$, the equation above holds for all integers, besides $n=0,$ which corresponds to the case $j= \frac{1}{2}$ in the sum. So we get
\begin{align*}
    \lim_{\overline{\tau} \to -i\infty} R(dz-\ell\tau,d\tau)&= 2 q^{-\frac{2k-1}{8} + \frac{\ell}{2}} \zeta^{- \frac{2k-1}{2}}. \qedhere
\end{align*}
\end{proof}
Combining the results of the previous lemmas, we get the following:
\begin{lemma}\label{lem:limitH}
\textit{We have}
\begin{align*}
    \lim_{\overline{\tau} \to -i\infty}\left(\zeta^{-\frac{1}{2}}-\zeta^{\frac{1}{2}} \right)  H_k(z,\tau)=(1-\zeta) \theta_{1,2k-1}(\tau) + \zeta(1-\zeta) \sum_{\ell=0}^{k-3} \zeta^k \theta_{2\ell+3,2k-1}(\tau). 
\end{align*}
\end{lemma}
\begin{proof} We use Lemma \ref{lem:limitR} to get
\begin{align*}
    \lim_{\overline{\tau} \to -i\infty} H_k(z,\tau)&= \lim_{\overline{\tau} \to -i\infty} \frac{i}{2} \sum_{\ell=0}^{k-2} \zeta^\ell \theta(\ell\tau,(2k-1)\tau)R((2k-1)z-\ell\tau,(2k-1)\tau)
    \\&= i \sum_{\ell=k}^{2k-2} \zeta^{\ell-\frac{2k-1}{2}} \theta(\ell\tau,(2k-1)\tau) q^{-\frac{2k-1}{8}+\frac{\ell}{2}}. 
\end{align*}
 Next, using Lemma \ref{lem:theta} and an index shift gives
\begin{align*}
       \lim_{\overline{\tau} \to -i\infty} H_k(z,\tau)&= \sum_{\ell=k}^{2k-2} \zeta^{\ell-\frac{2k-1}{2}}  \theta_{2\ell-(2k-1),2k-1}(\tau) 
    = \sum_{\ell=0}^{k-2} \zeta^{\ell+ \frac{1}{2}} \theta_{2\ell+1,2k-1}(\tau)
\end{align*}
Multiplying this with $\left(\zeta^{-\frac{1}{2}}-\zeta^{\frac{1}{2}} \right)$ and doing some algebraic manipulations, we get
\begin{align*}
      \sum_{\ell=0}^{k-2} \zeta^{\ell}(1-\zeta) \theta_{2\ell+1,2k-1}(\tau)  &=  (1-\zeta) \theta_{1,2k-1}(\tau) + \sum_{\ell=1}^{k-2} \zeta^\ell \left(1-\zeta \right) \theta_{2\ell+1,2k-1}(\tau) 
     % \\ &= (1-\zeta) \theta_{1,2w-1}(\tau) + \sum_{k=0}^{w-3} \zeta^k\left(\zeta- \zeta^2 \right) \theta_{2k+3,2w-1}(\tau) 
      \\ &= (1-\zeta) \theta_{1,2k-1}(\tau) + \zeta(1-\zeta) \sum_{\ell=0}^{k-3} \zeta^\ell \theta_{2\ell+3,2k-1}(\tau). 
\end{align*}
This yields the claim.
\end{proof}
Using the result of the lemma above, we can now determine the holomorphic part of \\ $\frac{-2i\sin(\pi z)}{(q)_\infty} \widehat{A}_{2k-1}(z,\tau)$ in $\tau.$ This is the content of the following lemma:
\begin{lemma}\label{lem:lim}
\textit{We have}
\begin{align*}
    \lim_{\overline{\tau} \rightarrow -i\infty} \frac{-2i\sin(\pi z)}{(q)_\infty}\widehat{A}_{2k-1}(z,\tau)
    %&=    FG_w(z,\tau) -\frac{2i\sin(\pi z)}{(q)_\infty} \lim_{\overline{\tau} \rightarrow -i\infty} H_w(z,\tau)
     %\\ &\hspace{0.75cm}+\frac{\zeta}{(q)_\infty} \theta_{1,2w-1}(\tau) - \frac{\zeta(1-\zeta)}{(q)_\infty}  \sum_{k=0}^{w-3} \zeta^k \theta_{2k+3,2w-1}(\tau) \\
      &= FG_k(\zeta,q) + \frac{\theta_{1,2k-1}(\tau)}{(q)_\infty} . 
\end{align*} 
\end{lemma}
%As we will use this in a moment, the right-hand side is exactly the term we use in Section 3 in order to define the family $\{f_{w,k}\}_{k \in \mathbb{N}}.$ 
\subsection{Pólya cycle index polynomials.} In order to show that the $k$-rank moments can be expressed in terms of traces of the family $f_k=\{f_{k,j}\}_{j\in \mathbb{N}}$ as stated in Theorem \ref{thm:main}, we will make use of the Pólya cycle index polynomials in the case of the symmetric group $S_n $ of the symbols $x_1,...,x_n$ as in \cite{PCIP}  and \cite[Lemma 2.1]{amdeberhan2025tracespartitioneisensteinseries}.
\begin{lemma}\label{lem:PCI} 
\textit{It holds that}
\begin{align*}
    \sum_{n =0}^\infty \sum_{\lambda \vdash n} \prod_{j=1}^n \frac{x_j^{\ell_j}}{\ell_j!}w^n= \exp \left( \sum_{j=1}^\infty x_j w^j \right).
\end{align*}
\end{lemma}
\subsection{A rank-crank type PDE} In \cite[Theorem 1.1]{atkin2002relationsrankscrankspartitions} Atkin and Garvan showed that the crank and rank generating functions are linked by a partial differential equation. Zwegers \cite{zwegers} then showed, that similar partial differential equations, so called rank-crank type partial differential equations, hold for $A_\ell(z,\tau).$ Namely \cite[Theorem 1.5]{zwegers} for any odd integer $\ell \geq 3$, there exist holomorphic modular forms $f_j$ of weight $j=4,6,8,...,\ell-1,$ such that 
\begin{align*}
    \left(\mathcal{H}^{\frac{\ell-1}{2}}+\sum_{j=0}^{\frac{\ell-5}{2}} f_{\ell-2j-1} \mathcal{H}^j  \right)A_{\ell}(z,\tau)=  (\ell-1)!\left( \frac{(q)_\infty}{-2i \sin(\pi z)}C(\zeta,q) \right)^\ell, \tag{2.12}
\end{align*}
with \begin{align*}
    \mathcal{H}_k:= \frac{\ell}{\pi i} \frac{\partial}{\partial \tau} + \frac{1}{(2\pi i)^2}\frac{\partial^2}{\partial z^2}+2\ell(2k-1) G_2(\tau)
\end{align*}
and
\begin{align*}
    \mathcal{H}^k:= \mathcal{H}_{2k-1} \mathcal{H}_{2k-3} \dots \mathcal{H}_3 \mathcal{H}_1.
\end{align*}
For our purposes, we will  mostly be interested in the case $\ell=5,$ for which the PDE reduces to \cite[Section 3]{zwegers}:
\begin{align*}
    \left( \mathcal{H}_3 \mathcal{H}_1 -\frac{220}{3} G_4(\tau)\right) A_5(z,\tau)  &=24   \left( \frac{(q)_\infty}{-2i \sin(\pi z)} C(\zeta,q) \right)^5 
    \\&= \frac{3i}{4\pi^5z^5} \exp\left(10 \sum_{k= 2}^\infty G_k(\tau)\frac{(2\pi i z)^k}{k!} \right).\tag{2.13}
\end{align*}

\section{Mock Eisenstein series from $k$-rank moments}
\subsection{Completion of $f_{k,j}$ and their modularity.} Following the ideas in Section 3 of \cite{bringmann2025mockeisensteinseriesassociated} we define the family $f_k=\{f_{k,j}\}_{j\in \mathbb{N}}$ using the holomorphic part calculated in Lemma \ref{lem:lim}:
\begin{align*}
    FG_k(\zeta,q) + \frac{\theta_{1,2k-1}(\tau)}{(q)_\infty} =: \frac{\sin(\pi z)}{\pi z(q)_\infty} \exp \left(2\sum_{j=1}^\infty f_{k,j}(\tau) \frac{(2\pi i z)^j}{j!} \right). \tag{3.1}
\end{align*}
We point out that the authors of \cite{bringmann2025mockeisensteinseriesassociated} had 1 in the place of $\tfrac{\theta_{1,2k-1}(\tau)}{(q)_\infty}$ for the case $k=2$ in their paper. Indeed, when  $k=2$ this quotient simplifies to $1.$

From this definition, we can already prove part (1) of Theorem \ref{thm:main}: 
\begin{proof}[Proof of Theorem \ref{thm:main} (1)]
Note, that the left-hand side of (3.1) is invariant under $\zeta \mapsto \zeta^{-1}$  (which corresponds to $z \mapsto -z,$) so $f_{k,j}\equiv 0$ for all odd $j.$ This can be seen from the definition of $FG_k(\zeta,q)$ and the fact that $\tfrac{\theta_{1,2k-1}(\tau)}{(q)_\infty}$ does not depend on $\zeta$ (or $z$). 

Regarding the the limiting behavior of $f_{k,j},$ we calculate
\begin{align*}
    &\frac{1}{2 \pi i z}\exp\left( 2 \sum_{j= 2 }^\infty f_{k,j}(\tau) \frac{(2\pi i z)^j}{j!} \right) =\frac{(q)_\infty FG_k(\zeta,q)+ \theta_{1,2k-1}(\tau)}{2i \sin(\pi z)} 
    \\ &\hspace{0.25cm}= \frac{\zeta^{\frac{1}{2}}}{\zeta-1} \left[ (q)_\infty FG_k(\zeta,q)+ \theta_{1,2k-1}(\tau)  \right] \xrightarrow{\tau \rightarrow i\infty} \frac{\zeta^{\frac{1}{2}}}{\zeta-1},
\end{align*}
because
 \begin{align*}
     \lim_{\tau \rightarrow i\infty} FG_k(\zeta,q)=0 \text{ and } \lim_{\tau \rightarrow i\infty} \theta_{1,2k-1}(\tau)=1.
 \end{align*} 
 Now using Lemma \ref{lem:bern} in our previous calculation gives the claim. 
 \end{proof}
 
Again following the ideas of Bringmann, Pandey and van Ittersum in \cite{bringmann2025mockeisensteinseriesassociated}, we define the function
\begin{align*}
    K^{\circ}_k(z,\overline{z},\tau,\overline{\tau}) &:=-2 i \sin(\pi z)q^\frac{1}{24}K_k(z,\tau)e^{-4(2k-1)\pi^2G_2(\tau)z^2} =\frac{-2i \sin(\pi z)}{(q)_\infty} \widehat{A}_{2k-1}(z,\tau) \\
    &\,=FG_k(\zeta,q)- \frac{2i\sin(\pi z)}{(q)_\infty}H_k(z,\tau) \\&\hspace{1cm}+ \frac{\zeta}{(q)_\infty} \theta_{1,2k-1}(\tau) -  \frac{\zeta(1-\zeta)}{(q)_\infty} \sum_{m=0}^{k-3} \zeta^m \theta_{2m+3,2k-1}(\tau)
\end{align*}
 Utilizing Lemma \ref{lem:limitH}, we have
\begin{align*}
    \lim_{\overline{\tau}\rightarrow -i\infty}  K^{\circ}_k(z,\overline{z},\tau,\overline{\tau})=FG_k(\zeta,q) + \frac{\theta_{1,2k-1}(\tau)}{(q)_\infty}.
\end{align*}

Next, we use this function, in order to define the quasi-completions $f^*_{k,j}$ of $f_{k,j} $ via:
\begin{align*}
    \frac{(q)_\infty K^\circ_k(z,\overline{z},\tau,\overline{\tau})}{2i \sin(\pi z)}=: \frac{1}{2\pi i z} \exp\left(2 \sum_{j,\ell \geq 0} f^*_{k,j,\ell}(\tau,\overline{\tau}) \frac{(2\pi i z)^j}{j!}\frac{(2\pi i \overline{z})^\ell}{\ell!}\right). \tag{3.2}
\end{align*}
Furthermore, we denote by $K^\text{*}_k(z,\tau,\overline{\tau})$ the constant term in the Taylor expansion in $\overline{z}$ of $K^\circ_k(z,\overline{z},\tau,\overline{\tau})$ and
\begin{align*}
    f_{k,j}^*(\tau):=f_{k,j,0}^*(\tau,\overline{\tau}).
\end{align*}
With these definitions, we have
\begin{align*}
    \mathbb{F}_k(z,\tau):= \frac{(q)_\infty K^\text{*}_k(z,\tau,\overline{\tau})}{2i \sin(\pi z)}=\frac{1}{2\pi i z} \exp \left( 2\sum_{j= 0}^\infty f_{k,j}^*(\tau) \frac{(2\pi i z)^j}{j!} \right). \tag{3.3}
\end{align*}
With this, $f_{k,j}^*(\tau)\equiv 0 $ if $j $ is odd. Having stated these definitions, we can prove part (2) of Theorem \ref{thm:main}. To do so, we will prove the following lemma:
\begin{lemma}\label{lem:comp}
\textit{For all} $ \left( {a\atop c} {b\atop d} \right)\in \text{SL}_2(\mathbb{Z}),$ \textit{it holds that 
\begin{align*}
    f^*_{k,j,\ell}\left(\frac{a\tau+b}{c\tau+d},\frac{a\overline{\tau}+b}{c\overline{\tau}+d}\right) =
    \begin{cases}
        (c\tau+d)^j(c\overline{\tau}+d)^{\ell}f^*_{k,j,\ell}(\tau,\overline{\tau}) 
        & \text{if } (j,\ell) \neq (2,0),
        \\ (c\tau+d)^2 f^*_{k,2,0} (\tau,\overline{\tau}) +\frac{(2k-1)ic}{4\pi}(c\tau+d)
        & \text{if } (j,\ell)=(2,0).
    \end{cases}
\end{align*} }
\end{lemma}
\begin{proof} By our previous definition of $f_{k,j,\ell}^*$ and Lemma \ref{lem:trans}, we have
\begin{align*}
    &\frac{c\tau+d}{2\pi i z} \exp \left( 2 \sum_{j,\ell \geq 0} f_{k,j,\ell}^* \left( \frac{a\tau+b}{c\tau+d},\frac{a\overline{\tau}+b}{c\overline{\tau}+d}\right) \frac{\left(\frac{2\pi i z}{c\tau+d} \right)^j}{j!}  \frac{\left(\frac{2\pi i \overline{z}}{c\overline{\tau}+d}\right)^\ell}{\ell!}\right)
    \\  &\hspace{0.5cm}=\eta \left(\frac{a\tau+b}{c\tau+d} \right) K_k\left(\frac{z}{c\tau+d} , \frac{a\tau+b}{c\tau+d} \right) e^{-4(2k-1) \pi^2 \left( \frac{z}{c\tau+d}\right)^2 G_2\left( \frac{a\tau+b}{c\tau+d}\right) }
    \\ &\hspace{0.5cm}= (c\tau+d) \eta(\tau) K_k(z,\tau) e^{-4(2k-1) \pi^2 z^2 \left(G_2(\tau) + \frac{ic}{4\pi (c\tau+d)} \right) } \\
    &\hspace{0.5cm}=\frac{c\tau+d}{2\pi i z} \exp\left( 2 \left( \frac{(2k-1)ic}{4\pi (c\tau+d)} \frac{(2\pi i z)^2}{2!}+\sum_{j,\ell \geq 0} f_{k,j,\ell}^*(\tau,\overline{\tau}) \frac{(2\pi i z)^j}{j!} \frac{(2\pi i \overline{z})^\ell}{\ell!}\right) \right).
\end{align*}
Which establishes the transformation formula for $f_{k,j,\ell}^*.$
\end{proof}
By the definition of the $f_{k,j}^*, $ Lemma \ref{lem:comp} implies the transformation formula in Theorem \ref{thm:main} (2).  In order to complete the proof of Theorem \ref{thm:main} (2), we have to verify the following lemma:
\begin{lemma}\label{lem:limitf}
\textit{It holds that}
\begin{align*}
    \lim_{\overline{\tau} \to -i\infty} f_{k,j}^*(\tau)=f_{k,j}(\tau).
\end{align*}
\end{lemma}
\begin{proof} Once again by the definitions of $f_{k,j,\ell}^*$ and $f_{k,j}^*,$ we have
\begin{align*}
    &\hphantom{=}\frac{1}{2\pi i z} \lim_{\overline{\tau} \rightarrow -i\infty} \exp\left( 2 \sum_{j,\ell \geq 0} f_{k,j,\ell}^*(\tau,\overline{\tau})\frac{(2\pi i z)^j}{j!} \frac{(2\pi i \overline{z})^\ell}{\ell!} \right) 
    \\ &= \frac{(q)_\infty}{2 i \sin(\pi z)} \lim_{\overline{\tau} \rightarrow -i\infty}  K^\circ_k(z,\overline{z},\tau, \overline{\tau})
    = \frac{(q)_\infty}{2i\sin(\pi z)} FG_k(\zeta,q) + \frac{\theta_{1,2k-1}(\tau)}{2i\sin(\pi z)}
    \\&=\frac{1}{2\pi i z} \exp\left(2 \sum_{j=0}^\infty f_{k,j}(\tau) \frac{(2\pi i z)^j}{j!}\right).
\end{align*}
By continuity of the exponential function, comparing the coefficients yields the desired result. 
\end{proof}

Lastly, we define a family $\widehat{f}_k:=\{\widehat{f}_{k,j}\}_{j\in \mathbb{N}}$ via
\begin{align*}
    \widehat{\mathbb{F}}_k(z,\tau):= \frac{(q)_\infty K^\text{*}_k(z,\tau,\overline{\tau})e^{-\frac{(2k-1) \pi z^2}{2v}}}{2i \sin(\pi z)} =: \frac{1}{2\pi i z} \exp \left( 2 \sum_{j =1}^\infty \hat{f}_{k,j}(\tau) \frac{(2\pi i z)^j}{j!}\right). \tag{3.4}
\end{align*}
\begin{lemma}
\textit{The $\widehat{f}_{k,j}$ satisfy:
\begin{align*}
    \widehat{f}_{k,j}(\tau)=\begin{cases}
        f^*_{k,j}(\tau) & \text{ if }j\neq 2 \\
        f_{k,2}^*(\tau)+ \frac{2k-1}{8\pi v} & \text{ if } j=2 \tag{3.5}
    \end{cases}
\end{align*}
and for all} $\left({a \atop c} {b \atop d} \right)\in \text{SL}_2(\mathbb{Z}),$ \textit{we have
\begin{align*}
    \widehat{f}_{k,j}\left( \frac{a\tau+b}{c\tau+d} \right) = (c\tau+d)^j \widehat{f}_{k,j}(\tau).
\end{align*} }
\end{lemma}
\begin{proof} Arguing as in the proof of Lemma \ref{lem:limitf} yields the claim.
\end{proof}
\subsection{Limiting behavior of $f_{k,j}^*$ and $\widehat{f}_{k,j}$.} We start by studying the limiting behavior of the $\theta$-function, as $\tau$ approaches $i \infty.$ 
\begin{lemma}\label{lem:thetaO}
\textit{We have for $k\geq 3$ and $\ell=0,...,2k-2$
\begin{align*}
    \theta(\ell\tau,(2k-1)\tau)= O\left(e^{\pi i \tau \left(\frac{2k-1}{4} -\ell\right)} \right)
\end{align*}
as $v \to \infty.$}
\end{lemma}
\begin{proof} We set $d:=2k-1$ and with this, have
\begin{align*}
    \theta(\ell\tau,d\tau)= i \sum_{n \in \mathbb{Z}} (-1)^ne^{\pi i \tau \left(dn^2+(d+2\ell)n+ \left( \frac{d+4\ell}{4} \right)\right)}.
\end{align*}
In order to find the dominant term, we have to find the minimal value the quadratic in the exponential attains on $\mathbb{Z}.$ By analytical methods, we find that the minimal value of $dn^2+(d+2\ell)n+\frac{(d+4\ell)}{4}$ on $\mathbb{R}$ is attained at $-\frac{\ell}{d}-\frac{1}{2}.$ Since quadratics are symmetric with respect to their vertex, the minimal value is attained at the closest integer to this point. Since $\ell$ ranges from $0$ to $d-1,$ this is always $n=-1$ (and also $n=0$ if and only if $\ell=0).$ This gives 
\begin{align*}
    d(-1)^2 + (d+2\ell)(-1) + \frac{d+4\ell}{4} = \frac{d}{4} -\ell.
\end{align*}
The corresponding term with this exponent is given by $ -i e^{\pi i \tau  \left( \frac{d}{4} -\ell\right)}.$
\end{proof}
Using this lemma, we can prove the following lemma:
\begin{lemma}\label{lem:limitHk}
\textit{We have for $k \geq 3,$ that} 
\begin{align*}
  \left(\zeta^{-\frac{1}{2}}-\zeta^{\frac{1}{2}} \right) \lim_{\tau \to i\infty} H_k(z,\tau)= 1- \zeta^{k-1}.
\end{align*}
\end{lemma}
\begin{proof} We again set $d:=2k-1$ and get that $ R(dz-\ell\tau,d\tau)$ is equal to
\begin{align*}
   &\sum_{n \in \mathbb{Z}+\frac{1}{2}} \left[\text{sgn}(n)-\text{sgn}\left(\left(n +\frac{y}{v} -\frac{\ell}{d} \right)\sqrt{2dv} \right) + \beta\left( 2dv\left( n+\frac{y}{v}-\frac{\ell}{d}\right)^2 \right)\right] \\&\hspace{1.5cm}(-1)^{n-\frac{1}{2}} q^{-\frac{dn^2}{2}} e^{-2\pi i n(dz-\ell\tau)}.
\end{align*}
As already shown in the proof of Lemma \ref{lem:limitR}, the difference of these two sgn-functions vanishes in the limit $\tau \to i \infty,$ if $\ell=0,...,k-1.$ For the other cases of $k,$ we get
\begin{align*}
    &\lim_{\tau \to i\infty} \sum_{n \in \mathbb{Z}+\frac{1}{2}} \left[\text{sgn}(n)-\text{sgn}\left(\left(n +\frac{y}{v} -\frac{\ell}{d} \right)\sqrt{2dv} \right)\right] (-1)^{n-\frac{1}{2}} q^{-\frac{dn^2}{2}} e^{-2\pi i n(dz-\ell\tau)}
    \\ &= 2 \zeta^{-\frac{d}{2}} e^{\left(-\frac{d}{4}+\ell\right)\pi i \tau}, \tag{3.6}
\end{align*}
arguing again as in the proof of Lemma \ref{lem:limitR}. For the terms involving $\beta$ we use an upper bound given in the proof of Lemma 1.8 in  \cite{zwegers2008mockthetafunctions}, namely
\begin{align*}
    \beta(m) \leq e^{-\pi m }   
\end{align*}
for real $m \geq 0 $ and get
 \begin{align*}
     &\sum_{n\in \mathbb{Z}+\frac{1}{2}} \left| \beta \left(2dv \left( n + \frac{y}{v} - \frac{\ell}{d}\right)^2 \right) (-1)^{n-\frac{1}{2}} q^{-\frac{dn^2}{2} } e^{-2\pi in (dz-\ell\tau)} \right| \\ 
     &\leq \sum_{n\in \mathbb{Z}+\frac{1}{2}}  e^{-2dv\pi  \left(n + \frac{y}{v} - \frac{\ell}{d} \right)^2} e^{d\pi v n^2}e^{-2\ell\pi nv}  e^{2d\pi ny} 
     %&\hspace{2cm}= \sum_{n\in \mathbb{Z}+\frac{1}{2}}  e^{-2dv\pi  \left(n^2 -\frac{2nk}{d} +\frac{k^2}{d^2} - \frac{2ny}{v}   - \frac{2ky}{dv}+\frac{y^2}{v^2}\right)} e^{d\pi v n^2 -2k\pi nv}  e^{2d\pi ny} \\
    = \hspace{-0.25cm}\sum_{n\in \mathbb{Z}+\frac{1}{2}} \hspace{-0.2cm} e^{-\pi v \left( dn^2 -2n\ell + \frac{2\ell^2}{d}\right) } e^{ 6d\pi ny   + 4\ell\pi y- \frac{2d\pi y^2}{v}}.
 \end{align*}
 Once again in order to determine the growth as $v\to \infty,$ we have to find the minimal value that $ dn^2-2n\ell+\frac{2\ell^2}{d^2} $ attains on $\mathbb{R}$ (it would suffice to find a more precise bound, by looking only at $n \in \mathbb{Z}+\frac{1}{2}$, but as we will see, the weaker bound is sufficient for our purposes). 
 The minimal value is attained at $n=-\frac{\ell}{d}$ and is $\frac{\ell^2}{d}.$ So in conclusion, we have
 \begin{align*}
     &\sum_{n\in \mathbb{Z}+\frac{1}{2}} \left| \beta \left(2dv \left( n + \frac{y}{v} - \frac{\ell}{d}\right)^2 \right) (-1)^{n-\frac{1}{2}} q^{-\frac{dn^2}{2} } e^{-2\pi in (dz-\ell\tau)} \right|  \leq C(y) e^{-\frac{2d\pi y^2}{v}} e^{-\pi v \frac{\ell^2}{d}},
 \end{align*}
 for large enough $v$. Next, we have
 \begin{align*}
   &\theta(\ell\tau,d\tau) R(dz-\ell\tau,d\tau) \\&=\theta(\ell\tau,d\tau)      \sum_{n \in \mathbb{Z}+\frac{1}{2}}  \beta\left( 2dv\left( n+\frac{y}{v}-\frac{\ell}{d}\right)^2 \right) (-1)^{n-\frac{1}{2}} q^{-\frac{dn^2}{2}} e^{-2\pi i n(dz-\ell\tau)}
    \\ &\hphantom{=}+ \theta(\ell\tau,d\tau)   \hspace{-0.25cm} \sum_{n \in \mathbb{Z}+\frac{1}{2}} \hspace{-0.1cm}\left[\text{sgn}(n)-\text{sgn}\left(\left(n +\frac{y}{v} -\frac{\ell}{d} \right)\sqrt{2dv} \right) \right] (-1)^{n-\frac{1}{2}} q^{-\frac{dn^2}{2}} e^{-2\pi i n(dz-\ell\tau)}.
 \end{align*}
 Using Lemma \ref{lem:thetaO} and our previous calculations, we get that the first term vanishes in the limit $\tau \to i\infty,$ as 
 \begin{align*}
    &\lim_{\tau \to i\infty} |\theta(\ell\tau,d\tau)|    \left|\sum_{n \in \mathbb{Z}+\frac{1}{2}}  \beta\left( 2dv\left( n+\frac{y}{v}-\frac{\ell}{d}\right)^2 \right) (-1)^{n-\frac{1}{2}} q^{-\frac{dn^2}{2}} e^{-2\pi i n(dz-\ell\tau)}\right|
    \\ &\leq \lim_{v \to \infty} C e^{-\pi  v \left( \frac{d}{4} -\ell\right)}  C(y) e^{-\pi v\frac{\ell^2}{d}} = \lim_{v\to \infty}\ C\cdot  C(y)  e^{-\pi v \left( \frac{\sqrt{d}}{2} - \frac{\ell}{\sqrt{d}}\right)^2 }=0,
 \end{align*}
as $ \left( \frac{\sqrt{d}}{2} - \frac{\ell}{\sqrt{d}}\right)^2 \geq 0$ and only vanishes if $\ell=\frac{d}{2}=k-\frac{1}{2},$ a value $\ell$ does not attain, as it is not an integer.

 For the terms of $R(dz-\ell\tau,d\tau)$ that are independent of $\beta$, we get by using Lemma \ref{lem:thetaO} and $(3.6)$ that
 \begin{align*}
    & \lim_{\tau \to i\infty} \hspace{-0.1cm} \theta(\ell\tau,d\tau)  \hspace{-0.25cm}    \sum_{n \in \mathbb{Z}+\frac{1}{2}} \hspace{-0.1cm }\left[\text{sgn}(n)-\text{sgn} \left(\left(n +\frac{y}{v} -\frac{\ell}{d} \right)\sqrt{2dv} \right) \right] \hspace{-0.05cm }(-1)^{n-\frac{1}{2}} q^{-\frac{dn^2}{2}} e^{-2\pi i n(dz-\ell\tau)}
   \\ &\hspace{1cm}= -2i\zeta^{-\frac{d}{2}},
 \end{align*}
 if $\ell=k,...,2k-2$ and the limit is $0$ if $\ell=0,...,k-1.$ This gives
 \begin{align*}
     \lim_{\tau \to i\infty} H_k(z,\tau) =  \frac{i}{2}\sum_{\ell=k}^{2k-2} -2i\zeta^{\ell} \zeta^{-\frac{d}{2}} = \sum_{\ell=k}^{2k-2} \zeta^{\ell-k+\frac{1}{2}} = \sum_{\ell=0}^{k-2} \zeta^{\ell + \frac{1}{2}}.
 \end{align*}
 By multiplying this with $\zeta^{\frac{1}{2}} - \zeta^{\frac{1}{2}} $ and using  the telescoping sum, we get the result. \end{proof}  
 Applying this result, we get the following corollary:
\begin{corollary} \textit{For $k \geq 3$ and all $j\geq 2,$ we have}
\begin{align*}
   \lim_{\tau \to  i\infty} f_{k,j}^*(\tau)= \lim_{\tau \to  i\infty} \hat{f}_{k,j}(\tau)=-\frac{B_j}{2j}.
\end{align*}
\end{corollary}
\begin{proof} For the $f_{k,j}^*$, it suffices to show that
\begin{align*}
   \lim_{\tau \to i\infty}  \frac{1}{2\pi iz} \left( \exp\left( 2\sum_{j= 1}^\infty f_{k,j}^*(\tau) \frac{(2\pi i z)^j}{j!} \right) - \exp\left(2 \sum_{j=1}^\infty f_{k,j}(\tau) \frac{(2\pi i z)^j}{j!} \right) \right)=0.
\end{align*}
By the definitions of the $f_{k,j}$ and $f_{k,j}^*$, we calculate
\begin{align*}
   \lim_{\tau\rightarrow i\infty} \frac{(q)_\infty \left(K^\text{*}_k(z,\tau)-FG_k(\zeta,q)-\frac{\theta_{1,2k-1}(\tau)}{(q)_\infty} \right)}{2i \sin(\pi z)}. \tag{3.7}
\end{align*}
Recall that
\begin{align*}
    K^o_k(z,\overline{z},\tau,\overline{\tau})&=FG_k(\zeta,q) - \frac{2i\sin(\pi z)}{(q)_\infty} H_k(z,\tau) \\&\hspace{1cm}+ \frac{\zeta}{(q)_\infty} \theta_{1,2k-1}(\tau) - \frac{\zeta(1-\zeta)}{(q)_\infty}  \sum_{\ell=0}^{k-3} \zeta^\ell \theta_{2\ell+3,2k-1}(\tau).
\end{align*}
Calculating the limits term wise is straight forward for all terms on the right-hand side. We have
\begin{align*}
     \lim_{\tau \to i \infty} FG_k(\zeta,q)=0
\end{align*}
and 
\begin{align*}
    \lim_{\tau \to i\infty} \theta_{1,2k-1}(\tau) =  \lim_{\tau \to i\infty}\theta_{2\ell+3,2k-1}(\tau) =1,
\end{align*}
for $\ell=0,...,k-3.$ By Lemma \ref{lem:limitHk} we also have
\begin{align*}
       \left(\zeta^{-\frac{1}{2}} - \zeta^{\frac{1}{2}} \right) \lim_{\tau \to i \infty} H_k(z,\tau)= 1- \zeta^{k-1}.
\end{align*}
So
\begin{align*}
    \lim_{\tau \to i\infty} K^\circ_k(z,\overline{z},\tau,\overline{\tau})=1.
\end{align*}
With the definition of $K^*_k,$ we get
\begin{align*}
    \lim_{\tau \to i\infty} \frac{(q)_\infty \left( K^\text{*}_k(z,\tau)-FG_k(\zeta,q)-\frac{\theta_{1,2k-1}(\tau)}{(q)_\infty} \right)}{2i \sin(\pi z)}=0,
\end{align*}
which yields the claim. The limit for $\hat{f}_{k,j}$ now follows directly from the limit of the $f_{k,j}^*$ and the definition of $\hat{f}_k.$ 
\end{proof}
\subsection{$k$-rank moments as traces.} In this section, we will show the following lemma: 
\begin{lemma} \textit{With the family $f_k=\{  f_{k,j}\}_{j \in \mathbb{N}}$ as defined in (3.1), we have}
\begin{align*}
    \sum_{j=0}^\infty R_{k,j}(q) \frac{z^j}{j!}= \frac{2 \sinh(\frac{z}{2})}{z(q)_\infty} \sum_{j=0}^\infty \text{Tr}_j(\phi,f_k;\tau) z^j,
\end{align*}
\textit{with $\phi$ as in (1.3).}
\end{lemma}
\begin{proof} Using Lemma \ref{lem:PCI} with 
$\omega=2\pi i z$ and $x_j=\frac{2}{j!} f_{k,j},$ we find that
\begin{align*}
    FG_k(\zeta,q)= \frac{\sin(\pi z)}{\pi z (q)_\infty} \sum_{j=0}^\infty \text{Tr}_j(\phi,f_k;\tau)(2\pi i z)^j.
\end{align*}
Now substituting $z$ by $\frac{z}{2\pi i}$ in 
\begin{align*}
    FG_k(\zeta,q)=\sum_{j= 0}^\infty R_{k,j}(q) \frac{(2\pi iz)^j}{j!} \tag{3.8}
\end{align*}
and using $\sin \left( \frac{z}{2i}\right)= -i\sinh \left( \frac{z}{2} \right)$ gives
\begin{align*}
    \sum_{j=0}^\infty R_{k,j}(q) \frac{z^j}{j!}&=\frac{2i\sin\left( \frac{z}{2i}\right)}{z(q)_\infty} \sum_{j=0}^\infty \text{Tr}_j(\phi,f_k;\tau) z^j= \frac{2\sinh(\frac{z}{2})}{z(q)_\infty} \sum_{j=0}^\infty \text{Tr}_j(\phi,f_k;\tau)z^j. \qedhere
\end{align*}
\end{proof}

\section{An algebra closed under differentiation}
\subsection{Rewriting terms as Laurent series in $z$.}
In this section, we will deviate from our previous notation of $f_{k,j},$ because we will only be concerned with the case $k=3.$ Suppressing this index will make the resulting formulas somewhat more compact. To prevent confusion, the reader will be reminded, once we switch back to the established notation.

 We want to use the rank-crank type partial differential equation (2.12) in order to study the $f_{j}.$ For this we rewrite by using (3.1)
\begin{align*}
    A_5(z,\tau)&= \frac{(q)_\infty FG_3(\zeta,q) +\zeta \theta_{1,5}(\tau) -\zeta(1-\zeta) \theta_{3,5}(\tau) }{-2 i \sin(\pi z)} \\
     &= \frac{(q)_\infty FG_3(\zeta,q) + \theta_{1,5}(\tau)-  \theta_{1,5}(\tau)   }{-2 i \sin(\pi z)}  - \frac{1}{2i}  \frac{\zeta}{\sin(\pi z)} \theta_{1,5}(\tau)- \zeta^{\frac{3}{2}} \theta_{3,5}(\tau) \\
     &= -\frac{1}{2\pi i z} \exp\left(2 \sum_{j= 2}^\infty f_j(\tau) \frac{(2\pi i z)^j}{j!}\right) +\frac{\theta_{1,5}(\tau)}{2i\sin(\pi z)}  - \frac{1}{2i}  \frac{\zeta}{\sin(\pi z)} \theta_{1,5}(\tau) \\&\hspace{1cm}- \zeta^{\frac{3}{2}} \theta_{3,5}(\tau).
\end{align*}

Now,  to make the process of applying the differential operator
less challenging, we will rewrite all the terms of (2.13) as Laurent series in $z$ and then apply the operator to each term individually.
Once again using (2.2) as well as Lemma \ref{lem:PCI} with $\omega= 2\pi i z$ and $x_j= \frac{10G_j(\tau)}{j!}$, we get using $\lambda=(1^{m_1},...,n^{m_n}),$
\begin{align*}
    24\left( \frac{(q)_\infty}{-2i \sin(\pi z)}C(\zeta,q) \right)^5
    &=\frac{3i}{4\pi^5z^5} \exp\left(10 \sum_{j= 2}^\infty G_j(\tau)\frac{(2\pi i z)^j}{j!} \right)
    \\&=\frac{3i}{4\pi^5z^5} \sum_{n = 0}^\infty \sum_{\lambda \vdash n} \left[ \prod_{j=1}^n \left( \frac{10 G_j(\tau)}{j!} \right)^{m_j} \frac{1}{m_j!}   \right] (2\pi i z)^n 
    \\&=\frac{3i}{4} \sum_{n = 0}^\infty \sum_{\lambda \vdash n} \left[ \prod_{j=1}^n \left( \frac{10 G_j(\tau)}{j!} \right)^{m_j} \frac{1}{m_j!}   \right] (2 i)^n ( \pi z)^{n-5}
\end{align*}
and for $\omega= 2\pi i z$ and $x_j=\frac{2f_j(\tau)}{j!}$ 
\begin{align*}
\frac{1}{2\pi i z} \exp\left(2 \sum_{j= 2}^\infty f_{j}(\tau) \frac{(2\pi i z)^j}{j!}\right) &= \frac{1}{2\pi i z} \sum_{n= 0}^\infty \sum_{\lambda \vdash n } \left[\prod_{j=1}^n \left(  \frac{2f_{j}(\tau)}{j!}\right)^{m_j} \frac{1}{m_j!} \right] (2\pi i z)^n  \\
&= \sum_{n = 0}^\infty \sum_{\lambda \vdash n } \left[\prod_{j=1}^n \left(  \frac{2f_{j}(\tau)}{j!}\right)^{m_j} \frac{1}{m_j!} \right] (2 \pi i z)^{n-1} .
\end{align*}
By the power series expansion of the exponential function,  
\begin{align*}
    \zeta^{1.5} \theta_{3,5}(\tau)= \sum_{k=0}^\infty \frac{(3\pi i z)^k}{k!} \theta_{3,5}(\tau).
\end{align*}
To rewrite the remaining terms, we will use (2.4) as well as $  B_n \left(\frac{1}{2}\right)= -B_n (1-2^{1-n}),$ 
 and obtain
\begin{align*}
    \frac{ \theta_{1,5}(\tau)}{\sin(\pi z)}= \sum_{n=0}^\infty 2 B_{2n} \frac{(-1)^n(1-2^{2n-1})}{(2n)!} (\pi z)^{2n-1} \theta_{1,5}(\tau)
\end{align*}
as well as
\begin{align*}
    \frac{\zeta}{\sin(\pi z)}\theta_{1,5}(\tau)= \sum_{n=0}^\infty B_n\left(\frac{3}{2}\right) \frac{(2i)^n}{n!} (\pi z)^{n-1} \theta_{1,5}(\tau).
\end{align*}
\newpage
\subsection{Applying the operator to the terms.} We have
\begin{align*}
    &\hphantom{=} \mathcal{H}_3 \mathcal{H}_1  \left(-\sum_{n =0}^\infty \sum_{\lambda \vdash n } \left[\prod_{k=1}^n \left(  \frac{2f_k(\tau)}{k!}\right)^{m_k} \frac{1}{m_k!} \right] (2 \pi i z)^{n-1}  \right)
    \\ &= \frac{25}{\pi^2 } \sum_{n = 6}^\infty \sum_{\lambda \vdash n} \sum_{\ell=1}^n \sum_{\substack{j=1 \\j\neq l}}^n \frac{m_jm_\ell 2^{m_j+m_\ell} f_j'(\tau)f_j(\tau)^{m_j-1}f_\ell'(\tau)f_\ell(\tau)^{m_\ell-1}}{j!^{m_j} \cdot m_j!  \cdot \ell!^{m_\ell} \cdot  m_k!} 
    \left[\prod_{\substack{k=1\\l \neq k \neq j}}^n \left( \frac{2 f_k(\tau)}{k!}\right)^{m_k} \frac{1}{m_k!} \right](2\pi i z)^{n-1} \\
    &\hspace{1cm}+\frac{25}{\pi^2}\sum_{n = 2}^\infty \sum_{\lambda \vdash n} \sum_{\ell=1}^n \frac{m_l2^{m_l} f_\ell''(\tau)f_\ell(\tau)^{m_\ell-1}}{\ell!^{m_\ell} \cdot m_\ell!} \left[\prod_{\substack{k=1\\ k \neq l}}^n \left(\frac{2f_k(\tau)}{k!}\right)^{m_k} \frac{1}{m_k!}  \right](2\pi i z)^{n-1}\\
    &\hspace{1cm}+ \frac{25}{\pi^2} \sum_{n = 4}^\infty \sum_{\lambda \vdash n} \sum_{\ell=1}^n\frac{m_\ell(m_\ell-1) 2^{m_\ell} f_\ell'(\tau)^2f_\ell(\tau)^{m_\ell-2}}{\ell!^{m_\ell} \cdot m_\ell!}
    \left[\prod_{\substack{k=1\\ k \neq l}}^n \left(\frac{2f_k(\tau)}{k!}\right)^{m_k} \frac{1}{m_k!} \right]  (2\pi i z)^{n-1}\\
    &\hspace{1cm}- \frac{10}{\pi i }\sum_{n = 2}^\infty \sum_{\lambda \vdash n} \sum_{\ell=1}^n \frac{m_\ell2^{m_\ell} f_\ell'(\tau)f_\ell(\tau)^{m_\ell-1}}{\ell!^{m_\ell} \cdot m_\ell!} \left[\prod_{\substack{k=1\\ k \neq l}}^n \left(\frac{2f_k(\tau)}{k!}\right)^{m_k} \frac{1}{m_k!}  \right](n-1)(n-2)(2\pi i z)^{n-3}\\
    &\hspace{1cm}-\frac{50}{\pi i} G_2'(\tau)  \sum_{n = 0}^\infty \sum_{\lambda \vdash n } \left[\prod_{k=1}^n  \left(  \frac{2f_k(\tau)}{k!}\right)^{m_k} \frac{1}{m_k!} \right] (2 \pi i z)^{n-1}   \\
    &\hspace{1cm}- \frac{300}{\pi i } G_2(\tau) \sum_{n = 2}^\infty \sum_{\lambda \vdash n } \sum_{\ell=1}^n \frac{m_\ell2^{m_\ell}f_\ell'(\tau)f_\ell(\tau)^{m_\ell-1}}{\ell!^{m_\ell} \cdot m_\ell!}
    \left[\prod_{\substack{k=1\\k \neq l}}^n \left(  \frac{2f_k(\tau)}{k!}\right)^{m_k} \frac{1}{m_k!} \right] (2 \pi i z)^{n-1}  \\
    &\hspace{1cm}- \sum_{n = 2}^\infty \sum_{\lambda \vdash n }\left[\prod_{k=1}^n \left(  \frac{2f_k(\tau)}{k!}\right)^{m_k} \frac{1}{m_k!} \right] (n-1)(n-2)(n-3)(n-4) (2 \pi i z)^{n-5} \\
    &\hspace{1cm}-60 G_2(\tau) \sum_{n = 2}^\infty \sum_{\lambda \vdash n } 
    \left[\prod_{k=1}^n \left(  \frac{2f_k(\tau)}{k!}\right)^{m_k} \frac{1}{m_k!} \right] (n-1)(n-2)(2 \pi i z)^{n-3} \\
    &\hspace{1cm}- 500 G_2^2(\tau) \sum_{n = 2}^\infty \sum_{\lambda \vdash n }     \left[\prod_{k=1}^n \left(  \frac{2f_k(\tau)}{k!}\right)^{m_k} \frac{1}{m_k!} \right] (2 \pi i z)^{n-1}.
\end{align*}
We have
\begin{align*}    
&\mathcal{H}_3 \mathcal{H}_1 \left( \sum_{k=0}^\infty \frac{(3\pi i z)^k}{k!}   \cdot \theta_{3,5}(\tau)\right)
    \\&= 25 \theta_{3,5}^{[2]}(\tau)  \cdot \sum_{k = 0}^\infty \frac{(3\pi i z)^k}{k!} 
    + \theta_{3,5}^{[1]}(\tau) \cdot \left( \frac{45}{2} +300 G_2(\tau) \right) \sum_{k= 0}^\infty \frac{(3\pi i z)^k}{k! }    \\
    &\hspace{1cm}+ \theta_{3,5}(\tau) \cdot  \left(\frac{50}{ \pi i}G_2'(\tau) +500G_2^2(\tau)  +135 G_2(\tau) +\frac{81}{16}\right) \sum_{k = 0}^\infty \frac{(3\pi i z)^k}{k!}  .
\end{align*}
Next, we have  
\begin{align*}    
&\mathcal{H}_3 \mathcal{H}_1 \bigg( \frac{1}{2i} \sum_{n \geq 0} B_n\left(\frac{3}{2}\right)\frac{(2i)^n}{n!} (\pi z)^{n-1} \cdot \theta_{1,5}(\tau)  \bigg)\\
    &= 25\sum_{n=0 }^\infty B_n\left(\frac{3}{2}\right) \frac{(2\pi i z)^{n-1}}{n!} \cdot \theta_{1,5}^{[2]}(\tau) 
    + 10 \sum_{n= 0}^\infty B_n\left(\frac{3}{2}\right) (n-1)(n-2) \frac{(2\pi i z)^{n-3}}{n!} \cdot  \theta_{1,5}^{[1]}(\tau)\\
    &\hspace{1cm}+ \frac{50}{\pi i}G_2'(\tau) \sum_{n= 0}^\infty B_n\left(\frac{3}{2}\right) \frac{(2\pi i z)^{n-1}}{n!} \cdot \theta_{1,5}(\tau)
    +300G_2(\tau)\sum_{n= 0}^\infty B_n\left(\frac{3}{2}\right) \frac{(2\pi i z)^{n-1}}{n!} \cdot \theta_{1,5}^{[1]}(\tau)\\
    &\hspace{1cm}+ \sum_{n = 0}^\infty B_n\left(\frac{3}{2}\right) (n-1)(n-2)(n-3)(n-4) \frac{(2\pi i z)^{n-5}}{n!} \cdot \theta_{1,5}(\tau)    
   \\&\hspace{1cm}+60G_2(\tau) \sum_{n = 0}^\infty B_n\left(\frac{3}{2}\right) (n-1)(n-2) \frac{(2\pi i z)^{n-3}}{n!} \cdot \theta_{1,5}(\tau)
    \\ &\hspace{1cm}+ 500 G_2^2(\tau) \sum_{n = 0}^\infty B_n\left(\frac{3}{2}\right) \frac{(2\pi i z)^{n-1}}{n!}\cdot \theta_{1,5}(\tau).
 \end{align*}
Lastly, we have  
\begin{align*} 
&\mathcal{H}_3 \mathcal{H}_1 \left( \sum_{n=0}^\infty B_{2n} \frac{(-1)^n(1-2^{2n-1})}{(2n)!} (\pi z)^{2n-1}   i \theta_{1,5}(\tau) \right)\\
    &=25 \sum_{n=0}^\infty B_{2n} \frac{(-1)^n(1-2^{2n-1})}{(2n)!} (\pi z)^{2n-1}   i\theta_{1,5}^{[2]}(\tau) \\
    &\hspace{1cm }- \frac{5}{2} \sum_{\substack{n=0\\ n \neq 1}}^\infty B_{2n} \frac{(-1)^n(1-2^{2n-1})}{2n \cdot (2n-3)!} (\pi z)^{2n-3}   i  \theta_{1,5}^{[1]}(\tau)\\
    &\hspace{1cm }+300G_2(\tau)\sum_{n=0}^\infty B_{2n} \frac{(-1)^n(1-2^{2n-1})}{(2n)!} (\pi z)^{2n-1}   i\theta_{1,5}^{[1]}(\tau)\\
    &\hspace{1cm }+ \frac{1}{16}  \sum_{\substack{n=0\\ n \neq 1\\n\neq 2}}^\infty B_{2n} \frac{(-1)^n(1-2^{2n-1})}{2n \cdot (2n-5)!} (\pi z)^{2n-5}   i  \theta_{1,5}(\tau) \\
    &\hspace{1cm }+ \frac{50}{\pi  } G_2'(\tau)  \sum_{n=0}^\infty B_{2n} \frac{(-1)^n(1-2^{2n-1})}{(2n)!} (\pi z)^{2n-1}  \theta_{1,5}(\tau) \\
    &\hspace{1cm }-15 G_2(\tau)\sum_{n=0}^\infty B_{2n} \frac{(-1)^n(1-2^{2n-1})}{2n \cdot (2n-3)!} (\pi z)^{2n-3}   i\theta_{1,5}(\tau) \\
    &\hspace{1cm }+ 500G_2^2(\tau) \sum_{n=0}^\infty B_{2n} \frac{(-1)^n(1-2^{2n-1})}{ (2n)!} (\pi z)^{2n-1}    i \theta_{1,5}(\tau).
\end{align*}

We will refrain from stating the complete differential equation in the terms we just calculated, but rather refer to them, as we need them for further calculations. 
\subsection{Proof of Theorem \ref{thm:main} (3)}   Next, we will provide a proof for a weaker version of Theorem \ref{thm:main} (3) and will then refine the result, by utilizing a differential equation for the $\theta$-series. We start with the following lemma:
\begin{lemma} \label{lem:alge}
\textit{The algebra
\begin{align*}
    \mathbb{F}:=\mathbb{C}[f_2,f_2',f_4,f_4',\dots,G_2,G_4,G_6,\dots, \theta_{1,5},\theta_{1,5}^{[1]},\dots,\theta_{3,5},\theta_{3,5}^{[1]},\dots] 
\end{align*}
is closed under $D.$}
\end{lemma}
\begin{proof} 
It is well-known, that the statement holds for the algebra 
\begin{align*}
    \mathbb{C}[G_2,G_4,G_6,\dots].
\end{align*}
By (1.6), it also holds for the algebra
\begin{align*}
    \mathbb{C}\left[\theta_{1,5},\theta_{1,5}^{[1]},\dots,\theta_{3,5},\theta_{3,5}^{[1]},\dots\right].
\end{align*}
Now by the rules of differentiation, the algebra
\begin{align*}
    \mathbb{C}\left[G_2,G_4,G_6,\dots,\theta_{1,5},\theta_{1,5}^{[1]},\dots,\theta_{3,5},\theta_{3,5}^{[1]},\dots\right]
\end{align*}
is also closed under differentiation. We will show by induction on $n \in \mathbb{N},$ that $f_{2n}''$ is also contained in the algebra $ \mathbb{F}.$

For the base case $n=1$, we turn our attention to the differential equation (2.13) and isolate the coefficients of $z$ on both sides. The coefficient of $z$ on the right side of the differential equation is
\begin{align*}
    -i \pi  \left[ 1000 G_2^3(\tau)+ 100 G_2(\tau)G_4(\tau)+\frac{2}{3} G_6(\tau)\right].
\end{align*}
On the left hand side, the $z$-coefficient from $-\frac{220 G_4(\tau) A_5(z,\tau)}{3}$ is given by
\begin{align*}
    \frac{220}{3}G_4(\tau) \left[ 2\pi if_2(\tau) + \pi i  \cdot \theta_{1,5}(\tau) +3\pi i \cdot \theta_{3,5}(\tau) \right].
\end{align*}
New, we will look at the $z$ coefficients of \\ $\mathcal{H}_3\mathcal{H}_1 \left(-\sum_{n = 0}^\infty \sum_{\lambda \vdash n } \left[\prod_{k=1}^n \left(  \frac{2f_k(\tau)}{k!}\right)^{m_k} \hspace{-0.1cm}\frac{1}{m_k!} \right] (2 \pi i z)^{n-1}  \right)$. Those are
\begin{align*}
    &\hphantom{+} \frac{50}{\pi} i f_2''(\tau)  -  \left( 60\pi if_4(\tau) + 600f_2'(\tau)  \right)  G_2(\tau) - 10 f_4(\tau) - \frac{2}{3}\pi if_6(\tau) 
    \\&- \bigg( 100 G_2'(\tau)+ 120 f_2'(\tau) +40 \pi if_2^2(\tau)+20\pi if_4(\tau) +360\pi i  G_2(\tau) f_2(\tau) 
    \\&\hspace{1cm}+ 1000\pi i G_2^2(\tau) \bigg) f_2(\tau).    
\end{align*}
Next, we will turn to the coefficients of $z$ obtained by applying the differential operators to $\sum_{k=0}^\infty \frac{(3\pi i z)^k}{k!}   \cdot \theta_{3,5}(\tau).$ Here we get
\begin{align*}
&75\pi i \cdot \theta_{3,5}^{[2]}(\tau) + \!\! \left( \frac{135}{2} \pi i  +900 \pi i G_2(\tau) \right)\cdot \theta_{3,5}^{[1]}(\tau)   
   \\& \hspace{1cm}+  \left( \frac{243}{16} \pi i  +  1500\pi i G_2^2(\tau) +405 \pi i G_2(\tau) +150 G_2'(\tau) \right)\cdot \theta_{3,5}(\tau)    
\end{align*}
Lastly, we can add up the coefficients of $z$ given by the remaining two terms, because they share common factors:
\begin{align*}
    &25 \pi i \cdot \theta_{1,5}^{[2]}(\tau) + \left( 300\pi i G_2(\tau)  +\frac{5}{2} \pi i \right)\cdot  \theta_{1,5}^{[1]}(\tau)\\  
    &\hspace{1cm}+ \left( 500\pi i  G_2^2(\tau)  -15\pi i  G_2(\tau) +\frac{1}{16}  \pi i +50G_2'(\tau) \right) \cdot \theta_{1,5}(\tau).
\end{align*}
Here $f_2''$ appears only a single time, being multiplied by the factor $\frac{50i}{\pi}.$ Therefore, this equation can be rearranged in order to express $f_2''$ it terms of the elements of $\mathbb{F}.$ 

For the induction step, note that $f_{2n}''$ appears for the first time when considering the coefficients of $z^{2n-1}$  and it is only multiplied by a non-zero complex number. Now, by the induction hypothesis, all other terms in this equation are already elements of $\mathbb{F},$ so the equation for the coefficients of $z^{2n-1}$ proves that $f_{2n}'' $ is an element of $\mathbb{F}.$  
\end{proof}
We can refine this result, by applying a result from Zwegers:
\begin{proof}[Proof of Theorem \ref{thm:main} (3)]  In \cite[Lemma 1.6]{zwegers} Zwegers also gave a differential equations for  $\theta$-functions, namely
\begin{align*}
    \tilde{\theta}_{\ell,r}(\tau):= \sum_{n \in \mathbb{Z}} (-1)^n q^{\frac{\ell}{2}\left(n-\frac{1}{2} +\frac{r}{\ell} \right)^2},
\end{align*}
for $\ell \geq 3 \text{ odd}, r \in \mathbb{Z}, \tau \in \mathbb{H}.$
These functions satisfy the following differential equation:
\begin{align*}
    \left( D^{\frac{\ell-1}{2}} + \sum_{k=0}^{\frac{\ell-5}{2}} F_{\ell-2k-1}D^k\right) \tilde{\theta}_{\ell,r}= 0 , \tag{4.1}
\end{align*}
where $F_j$ are holomorphic modular forms of weight $j=4,6,8,...,\ell-1,$
\begin{align*}
   D_k:=\frac{1}{2\pi i} \frac{\partial}{\partial \tau} + 2kG_2 
\end{align*}
and
\begin{align*}
    D^k:= D_{2k-\frac{3}{2}} D_{2k-\frac{7}{2}} \dots D_{\frac{5}{2}} D_{\frac{1}{2}}.
\end{align*}
Substituting  $\ell=5$ and $r=3$ resp. $r=4$ into $\tilde{\theta}_{\ell,r}$  yields our functions $\theta_{1,5}$ resp. $\theta_{3,5}$ multiplied with some rational powers of $q$:
\begin{align*}
\tilde{\theta}_{5,3}(\tau)= q^\frac{1}{40} \theta_{1,5}(\tau),\hspace{0.5cm}
\tilde{\theta}_{5,4}(\tau)= q^\frac{9}{40} \theta_{3,5}(\tau).
\end{align*}
In the $\ell=5$ case, the differential equation simplifies to
\begin{align*}
    \left( D^2 + F_0 \right) \tilde{\theta}_{5,r}=0.
\end{align*}
Zwegers showed that $F_0(\tau)=-\frac{11}{15} G_4(\tau),$ so we get
\begin{align*}
    D^2 \tilde{\theta}_{5,3}(\tau)=\frac{11}{15} G_4(\tau) \tilde{\theta}_{5,3}(\tau)=\frac{11}{15} q^\frac{1}{40}G_4(\tau) \theta_{1,5}(\tau) .\tag{4.2}
\end{align*}
The left side of (4.2)  for $r=3$ is given by
\begin{align*}
    D^2 \left(\tilde{\theta}_{5,3}(\tau)\right)&= \left(\frac{1}{2\pi i} \frac{\partial}{\partial \tau} +5 G_2 (\tau) \right) \left( \frac{1}{2\pi i} \frac{\partial}{\partial \tau}+G_2(\tau) \right) \tilde{\theta}_{5,3}(\tau) 
    \\&=  \left(\frac{1}{2\pi i} \frac{\partial}{\partial \tau} +5 G_2 (\tau) \right)  \left[q^\frac{1}{40} \left( \frac{1}{2} \theta_{1,5}^{[1]}(\tau)  + \frac{1}{40} \theta_{1,5}(\tau) + G_2(\tau) \theta_{1,5}(\tau) \right)\right]   
    \\&=   \frac{1}{4}q^\frac{1}{40} \theta_{1,5}^{[2]}(\tau) + \left( \frac{1}{40}  +3 G_2(\tau) \right) q^\frac{1}{40}\theta_{1,5}^{[1]}(\tau)
    \\ &\hspace{0.5cm}+ \left( 5 G_2^2(\tau) + \frac{3}{20}  G_2(\tau)   + \frac{1}{1600} + \frac{1}{2\pi i}  G_2'(\tau)\right) q^\frac{1}{40}\theta_{1,5}(\tau).
\end{align*}
Now multiplying both sides of the resulting equation with $q^{-\frac{1}{40}}$, we get that $\theta_{1,5}^{[2]}(\tau)$ is an element of the algebra  $\mathbb{C}[G_2,G_4,\dots,\theta_{1,5},\theta_{1,5}^{[1]},\theta_{3,5},\theta_{3,5}^{[1]}]$. Analogously, a similar calculation gives the result for $\theta_{3,5}^{[2]}(\tau)$.   This proves the theorem. 
\end{proof}
In light of (3) of Theorem \ref{thm:bring} and (3) of Theorem \ref{thm:main}, one naturally wonders, if a similar result holds for the other ranks as well. We will sketch a proof of the corresponding theorem here. 
\begin{theorem}\label{thm:alge}
\textit{For $k \geq 3$ it holds that the algebra
\begin{align*}
    \mathbb{C}[f_{k,2}^{[j]},f_{k,4}^{[j]},,....,G_2,G_4,...,\theta_{2\ell+1,2k-1}^{[j]}]_{\substack{\ell=0,...,k-2 \\ j=0,...,k-2}}
\end{align*}
is closed under the action of $D.$ }
\end{theorem}
\begin{proof}[Sketch of proof:] As for the case $k=3$, the proof is done via induction over the even $n \in \mathbb{N}.$ Using the rank-crank type PDE (2.12) one obtains a partial differential equation of order $2k-2$. The derivative with respect to $\tau $ in this PDE is only taken up to the $(k-1)$-th order.  Rewriting $A_{2k-1}(z,\tau)$ as in the proof of  Lemma \ref{lem:alge}, then applying the differential operators, and comparing the coefficients as we did in the proof of Lemma \ref{lem:alge}, we get the 
$(k-1)$-th derivative of $f_{k,j}$ comes up the first time, when looking at the coefficient of $z^{j-1}$. As in the previous proof, this derivative is only multiplied by a non-zero complex number. Rearranging, and using the induction hypothesis, one gets that the algebra
\begin{align*}
    \mathbb{C}\left[f_{k,2}^{[j]},f_{k,4}^{[j]},...,G_2,G_4,\theta_{2\ell+1,2k-1}^{[m]} \right]_{\substack{\ell=0,...,k-2\\j=0,...,k-2\\ m\in \mathbb{N}_0}}
\end{align*}
is closed under differentiation. One quickly verifies that for $\ell=0,...,k-2,$ it holds that
\begin{align*}
    q^{-\frac{4\ell^2+4\ell+1}{8(2k-1)}}\theta_{2\ell+1,2k-1}(\tau)= \tilde{\theta}_{2k-1,\ell+k+1}(\tau).
\end{align*}
Using this and (4.1), one finds that the algebra
\begin{align*}
    \mathbb{C}\left[G_2,G_4,...,\theta_{2\ell+1,2k-1}^{[j]}\right]_{\substack{\ell=0,...,k-2\\j=0,...,k-2}}
\end{align*}
is closed under differentiation with respect to $\tau,$ by arguing the way we did in the end of proof of Theorem \ref{thm:main} (3). \end{proof}

\section{Formulas for $f_{k,j}$ in terms of divisor-like sums}
The methods we will use to prove the claims in this section are essentially those used by Bringmann, Pandey and van Ittersum in \cite{bringmann2025mockeisensteinseriesassociated}. In some cases small tweaks will be made to the argument, while in others the argument provided by them translates directly to the more general case here.

Before proving Theorem \ref{thm:rec} and \ref{thm:integral}, we need some auxiliary statements, which we will prove beforehand.  Starting with  the following lemma:
\begin{lemma}\label{lem:sums} 
\textit{ For $k\geq 3$ and $j \geq 1,$ we have}
\begin{align*}
    R_{k,j}(q)= \frac{2^{2-j}}{(q)_\infty} \sum_{\substack{\ell=2 \\ \ell \equiv j \text{ (mod 2)}}}^j \binom{j}{\ell-1} \left( g_{2,2k-1,\ell}(\tau) +\left(2^{\ell-1}-1 \right) \frac{B_\ell}{2\ell} \right).\tag{5.1}
\end{align*}
\textit{For $j=0,$ we have}
\begin{align*}
    R_{k,0}(q)= \frac{1}{(q)_\infty }  \sum_{n =1}^\infty (-1)^{n+1}\left(q^{\frac{(2k-1)n^2-n}{2}} + q^{\frac{(2k-1)n^2+n}{2}} \right)=\frac{1-\theta_{1,2k-1}(\tau)}{(q)_\infty }.   \tag{5.2}
\end{align*}
\end{lemma}
\begin{proof} The proof of the first claim follows the scaffold the proof of Lemma 4.1 of \cite{bringmann2025mockeisensteinseriesassociated}: If $j$ is odd, both sides of the equation are $0$, because
$N_{k}(m,n)=N_{k}(-m,n),$ hence $R_{k,j}(q)=0, $ and for all odd $j$ we have
\begin{align*}
    g_{2,2k-1,\ell}(\tau) +\left(2^{\ell-1}-1 \right) \frac{B_\ell}{2\ell}  =0,
\end{align*}
so the statement follows. Now, let $j \neq 0$ be even. We set $d:= 2k-1$ and get
\begin{align*}
    R_{k,j}(q) &= \sum_{n=0}^\infty  \sum_{m \in \mathbb{Z}} m^j N_k(m,n)q^n  
    \\ &= \sum_{m\in\mathbb{Z}} m^j \frac{1}{(q)_\infty}\sum_{n \geq 1} (-1)^{n+1} q^{\frac{n(dn-1)}{2}+|m|n} (1-q^n)
    \\ &= \frac{2}{(q)_\infty} \sum_{n = 1}^\infty (-1)^{n+1}q^{\frac{n(dn-1)}{2}} (1-q^n) \sum_{m = 0}^\infty m^j q^{mn}. \tag{5.3}
\end{align*}
Distinguishing between odd and even $n$, we get
\begin{align*}
    2^{j-1}(q)_\infty R_{k,j}(q)&= \sum_{n \geq 1, m \geq 0} (2m)^j \bigg(-q^{n(2dn+2m-1)} + q^{(2n-1)(dn+m-k)} \\
    &\hspace{4cm}+ q^{n(2dn+2m+1)} -q^{(2n-1)(dn+m-k+1)} \bigg) 
\end{align*}
Making the change of variables $m\mapsto m-dn$ gives
\begin{align*}
    &\sum_{m \geq dn \geq d} (2m-2dn)^j \left( -q^{n(2m-1)} + q^{(2n-1)(m-k)} + q^{n(2m+1)} -q^{(2n-1)(m-k+1)}  \right)
    \\ &= \sum_{m \geq dn \geq d} (2m-2dn)^j \left( -q^{n(2m-1)} + q^{n(2m+1)} \right)
    \\ &\hphantom{=}+ \sum_{n \geq dm \geq d} (2n-2dm)^j \left( q^{(n-k)(2m-1)} -q^{(n-k+1)(2m-1)}\right)
\end{align*}
when we interchange the roles of $m$ and $n$ in the second and fourth term.
Next we apply the following shifts: $m \mapsto m-1$ in the second term, $n\mapsto n+k$ in the third and $n \mapsto n+k-1$ in the fourth. This results in
\begin{align*}
    & -\sum_{m \geq dn \geq d} (2m-2dn)^j q^{n(2m-1)} 
    + \sum_{m-1\geq dn \geq d} (2m-2-2dn)^j q^{n(2m-1)}
    \\ &+ \sum_{n+k\geq dm \geq d } (2n+2k-2dm)^j q^{n(2m-1)}
    - \sum_{n+k-1\geq dm \geq d} (2n+2k-2-2dm)^j q^{n(2m-1)}
      \\ &= -\sum_{2m-1 \geq 2dn-1 \geq 2d-1} \hspace{-0.7cm}(2m-1-2dn+1)^j q^{n(2m-1)} 
    + \hspace{-0.7cm}\sum_{2m-1\geq 2dn+1 \geq 2d+1} \hspace{-0.7cm} (2m-1-2dn-1)^j q^{n(2m-1)}
    \\ &\hspace{0.6cm}+ \sum_{2n+1\geq d(2m-1) \geq d } \hspace{-0.7cm}\left(2n+1-d(2m-1)\right)^j q^{n(2m-1)}
    - \hspace{-0.7cm}\sum_{2n-1\geq d(2m-1) \geq d} \hspace{-0.7cm}\left(2n-1-d(2m-1) \right)^j q^{n(2m-1)},
\end{align*}
and
\begin{align*}
        &\sum_{\substack{m \geq 2dn+1 \geq 2d+1\\ m \text{ odd}}} \left( (m-2dn-1)^j - (m-2dn+1)^j \right) q^{nm} 
    \\&\hspace{1cm}+ \sum_{\substack{2n-1\geq dm\geq d\\ m \text{ odd}} } \left( (2n+1-dm)^j -(2n-1-dm)^j\right)  q^{nm}
    \\ &\hspace{0.5cm}=  \sum_{m \geq 2dn+1 \geq 2d+1} \left( (m-2dn-1)^j - (m-2dn+1)^j \right) q^{nm} 
   \\&\hspace{1.5cm} + \sum_{2n-1\geq dm\geq d } \left( (2n+1-dm)^j -(2n-1-dm)^j\right)  q^{nm}
\end{align*}
Now using the binomial theorem as in the proof of Lemma 4.1 of \cite{bringmann2025mockeisensteinseriesassociated}, we get the first claim.

For $j=0,$ one calculates 
\begin{align*}
    R_{k,0}(q) &=\sum_{n=0}^\infty \sum_{m\in \mathbb{Z}} N_k(m,n)q^n
    \\ &= \sum_{m \in \mathbb{Z}}  \frac{1}{(q)_\infty} \sum_{n =1}^\infty (-1)^{n+1}q^{\frac{n(dn-1)}{2}+|m|n} (1-q^n)
    \\ &=  \frac{1}{(q)_\infty} \sum_{n =1}^\infty (-1)^{n+1}q^{\frac{n(dn-1)}{2}} (1-q^n) \left( \sum_{m=0}^\infty \left(q^n\right)^m  + \sum_{m=1}^\infty \left(q^n\right)^m\right)
    \\&= \frac{1}{(q)_\infty }  \sum_{n =1}^\infty (-1)^{n+1}q^{\frac{n(dn-1)}{2}} (1-q^n)  \frac{1+q^n}{1-q^n}
    \\&= \frac{1}{(q)_\infty }  \sum_{n =1}^\infty (-1)^{n+1}\left(q^{\frac{n(dn-1)}{2}} + q^{\frac{n(dn+1)}{2}} \right).
\end{align*}
We can rewrite this sum as a sum ranging over $\mathbb{Z}$, if we correct for the missing term for $n=0,$ which is $-1$. With this we get the expression in terms of $\theta_{1,2k-1}(\tau).$ 
\end{proof}
With this we can now prove the following lemma:
\begin{lemma} \textit{For $k \geq 3,$ we have
\begin{align*}
    \frac{\pi z (q)_\infty}{\sin(\pi z)} FG_k(\zeta,q)&= \sum_{n = 1}^\infty \frac{n}{2^{n-2}}\left( g_{2,2k-1,n}(\tau) + \frac{\left(2^{n-1}-1\right)B_n}{2n}\right) \frac{(2\pi i z)^n}{n!}
    \\&\hspace{1cm}+(q)_\infty R_{k,0}(q)  \sum_{n=0}^\infty B_n\left(\frac{1}{2}\right)\frac{(2\pi i z)^n}{n!}.  \tag{5.4}
\end{align*}}
\end{lemma}
\begin{proof} It holds that
\begin{align*}
    FG_k(\zeta,q)= \sum_{j =0}^\infty R_{k,j}(q) \frac{(2\pi i z)^j}{j!}.
 \end{align*}
 Next, we have by Lemma \ref{lem:sums}
 \begin{align*}
     &\frac{\pi z (q)_\infty}{\sin(\pi z)} \left(FG_k(\zeta,q)- R_{k,0}(q) \right) \\ &\hspace{0.5cm}= \sum_{n = 0}^\infty B_n\left( \frac{1}{2}\right) \frac{(2 \pi i z)^n}{n!} \sum_{j =1}^\infty 2^{2-j} \hspace{-0.5cm} \sum_{\substack{\ell=1\\ \ell \not\equiv j \text{( mod 2}) }}^j\hspace{-0.5cm} \binom{j}{\ell}  \left( g_{2,2k-1,\ell+1}(\tau) +\frac{\left(2^{\ell}-1\right)B_{\ell+1}}{2(\ell+1)} \right) \frac{(2\pi i z)^j}{j!}
     \\ &\hspace{0.5cm}= \sum_{\ell = 1}^\infty \sum_{n = 0}^\infty \sum_{\substack{j = \ell \\ j \not\equiv \ell \text{ (mod 2)}}}^n \hspace{-0.5cm}\binom{n}{j} \binom{j}{\ell} B_{n-j}\left(\frac{1}{2} \right) 2^{2-j} \left(  g_{2,2k-1,\ell+1}(\tau) + \frac{\left(2^\ell-1 \right)B_{\ell+1}}{2(\ell+1)} \right) \frac{(2\pi i z)^n}{n!},
 \end{align*}
with the change of variables $n \mapsto n-j.$
Using  (2.5),  $B_m\left(\frac{1}{2}\right)=0$ if $m$ is odd and $m\neq 1$ and $B_1\left(\frac{1}{2}\right)=-\frac{1}{2},$ we get that for odd $l \in \mathbb{N}$ and $n \in \mathbb{N}$ we find
\begin{align*}
    \sum_{\substack{j=\ell\\ j \not\equiv \ell \text{ (mod 2)}}}^n \binom{n}{j} \binom{j}{\ell} B_{n-j}\left( \frac{1}{2} \right) 2^{\ell-j} = - \binom{n}{\ell} B_{n-\ell}(0)= \begin{cases}
    \frac{n}{2} & \text{\textit{if }} n=\ell+1\\
    0 & \text{otherwise.}
    \end{cases}
\end{align*}
Combining this with our previous calculations, we get
\begin{align*}
    \hspace{-0.3cm}\frac{\pi z (q)_\infty}{\sin(\pi z)} \left(FG_k(\zeta,q)-R_{k,0}(q) \right) = \sum_{n = 1}^\infty \frac{n}{2^{n-2}}\left(  g_{2,2k-1,n}(\tau) + \frac{\left( 2^{n-1}-1\right)B_n}{2n} \right) \frac{(2\pi i z)^n}{n!}.\hspace{-0.3cm} \tag{5.5}
\end{align*}
Lastly, because
\begin{align*}
    \frac{\pi z}{\sin(\pi z)} =\sum_{n = 0}^\infty B_n \left( \frac{1}{2} \right) \frac{(2\pi i z)^n}{n!}
\end{align*}
 (5.5) is equivalent to
\begin{align*}
    \frac{\pi z (q)_\infty}{\sin(\pi z)} FG_k(\zeta,q)&=\sum_{n = 1}^\infty \frac{n}{2^{n-2}}\left( g_{2,2k-1,n}(\tau) + \frac{\left(2^{n-1}-1\right)B_n}{2n}\right) \frac{(2\pi i z)^n}{n!}
    \\&\hspace{1cm}+(q)_\infty R_{k,0}(q)\sum_{n=0}^\infty B_n\left(\frac{1}{2}\right)  \frac{(2\pi i z)^n}{n!}. \qedhere
\end{align*}
\end{proof}
With these preliminaries, we can now provide a proof for Theorem \ref{thm:rec}:
\begin{proof}[Proof of Theorem \ref{thm:rec}:] Recall that we defined
\begin{align}
    \frac{\pi z (q)_\infty}{\sin(\pi z)} FG_k(\zeta,q) + \frac{\pi z \theta_{1,2k-1}(\tau)}{\sin(\pi z)} = \exp\left(2 \sum_{j = 1}^\infty f_{k,j}(\tau) \frac{(2\pi iz)^j}{j!} \right) \tag{5.6}
\end{align}
By our previous lemma and the calculations in its proof, we have
\begin{align*}
     \frac{\pi z (q)_\infty}{\sin(\pi z)} FG_k(\zeta,q)&=\sum_{n = 1}^\infty \frac{n}{2^{n-2}}\left( g_{2,2k-1,n}(\tau) + \frac{\left(2^{n-1}-1\right)B_n}{2n}\right) \frac{(2\pi i z)^n}{n!}
    \\&\hspace{1cm}+(q)_\infty R_{k,0}(q)\sum_{n=0}^\infty B_n\left(\frac{1}{2}\right)  \frac{(2\pi i z)^n}{n!}.
\end{align*}
and
\begin{align*}
    \frac{\pi z \theta_{1,2k-1}(\tau)}{\sin(\pi z)} = \theta_{1,2k-1}(\tau) \sum_{j = 0 }^\infty B_j\left(  \frac{1}{2}\right) \frac{(2\pi i z)^j}{j!}.
\end{align*}
So the left-hand side of $(5.6) $ is equal to 
\begin{align*}
    &\sum_{j=1}^\infty  \left( \frac{jg_{2,2k-1,j}(\tau)}{2^{j-2}} + \left(1-2^{1-j}\right) B_j \right)\frac{(2\pi i z)^j}{j!} 
    \\    &\hspace{1.5cm}+\left( (q)_\infty R_{k,0}(q) + \theta_{1,2k-1}(\tau)\right) \sum_{j=0}^\infty  B_j\left(  \frac{1}{2}\right) \frac{(2\pi i z)^j}{j!} 
    \\ &\hspace{0.5cm}= \sum_{j=1}^\infty \left( \frac{jg_{2,2k-1,j}(\tau)}{2^{j-2}} +\left(1-2^{1-j}\right) B_j \right) \frac{(2\pi i z)^j}{j!} +  \sum_{j=0}^\infty  B_j\left(  \frac{1}{2}\right) \frac{(2\pi i z)^j}{j!}
    \\ &\hspace{0.5cm}= 1+ \sum_{j=1}^\infty   \frac{jg_{2,2k-1,k}(\tau)}{2^{j-2}} \frac{(2\pi i z)^j}{j!} \tag{5.7},
\end{align*}
by (5.2) and the fact that $B_j\big(\frac{1}{2}\big)=-(1-2^{1-j})B_j$. Now taking the derivative with respect to $z$ of $(5.6),$ and using (5.7), we get
\begin{align*}
     \hspace{-0.3cm}2\pi i \sum_{j = 1}^\infty  \frac{jg_{2,2k-1,j}(\tau)}{2^{j-2}} \frac{(2\pi i z)^{j-1}}{(j-1)!}  
   \hspace{-0.1cm}= 4\pi i \sum_{j = 1}^\infty f_{k,j}(\tau) \frac{(2\pi i z)^{j-1}}{(j-1)!}\exp\hspace{-0.1cm}\left(2 \sum_{j = 1}^\infty f_{k,j}(\tau) \frac{(2\pi i z)^j }{j!} \right)\hspace{-0.3cm} \tag{5.8}
\end{align*}
Substituting (5.6) into (5.8)  gives:
\begin{align*}
    &2\pi i \sum_{j = 1}^\infty \frac{jg_{2,2k-1,j}(\tau)}{2^{j-2}} \frac{(2\pi i z)^{j-1}}{(j-1)!} 
=  4\pi i \hspace{-0.1cm} \sum_{j = 1}^\infty f_{k,j}(\tau) \frac{(2\pi i z)^{j-1}}{(j-1)!} \left(1 + \sum_{j = 1}^\infty \frac{jg_{2,2k-1,j}(\tau)}{2^{j-2}} \frac{(2\pi i z)^j}{j!} \right)    
    \\ &\hspace{1cm}= 4\pi i \sum_{j = 1}^\infty f_{k,j}(\tau) \frac{(2\pi i z)^{j-1}}{(j-1)!}
+ 4\pi i \sum_{n= 2}^\infty \sum_{j=2}^{n-1} \frac{f_{k,j}(\tau)}{(j-1)!}  \frac{(n-j)g_{2,2k-1,n-j}(\tau)}{2^{n-j-2} (n-j)!}   (2\pi i z)^{n-1}.    
\end{align*}
Comparing the coefficients of $(2\pi i z)^{j-1} $ gives the first equation.  For the second equation, multiply the exponential term in (5.8) to the other side and use Lemma \ref{lem:PCI} with $x_j=-\frac{2f_{k,j}}{j!}$ and $\omega= 2\pi i z .$ This yields
\begin{align*}
    2 \sum_{j = 1}^\infty f_{k,j}(\tau) \frac{(2\pi i z)^{j-1}}{(j-1)!} = \left[ \sum_{n= 1}^\infty \frac{n g_{2,2k-1,n}(\tau)}{2^{n-2}} \frac{(2\pi i z)^{n-1}}{(n-1)!}  \right] \sum_{m= 0 }^\infty \text{Tr}_m(\psi,f_k;\tau) (2\pi i z)^m.
\end{align*}
Now comparing the coefficients  of $(2\pi i z)^{j-1}$ gives the claim. 
\end{proof}

Lastly, we will provide the proof for Theorem \ref{thm:integral}:
\begin{proof}[Proof of Theorem \ref{thm:integral}:] By Lemma \ref{lem:bern}, we have
\begin{align*}
    \frac{\sin(\pi z)}{\pi z} = \exp\left( \sum_{j= 2}^\infty  \frac{B_j}{j} \frac{(2\pi i z)^j}{j!} \right).
\end{align*}
Using this and  (3.1) gives 
\begin{align*}
    \exp\left(2 \sum_{j= 1}^\infty \left(f_{k,j}(\tau) + \frac{B_j}{2j} \right) \frac{(2\pi i z)^j}{j!} \right)=  (q)_\infty FG_k(\zeta,q) + \theta_{1,2k-1}(\tau). 
 \end{align*}
If we now use (3.8) and replace $2\pi i z$ by $z$, we get 
\begin{align*}
    \exp\left(2 \sum_{j= 1}^\infty \left(f_{k,j}(\tau) + \frac{B_j}{2j} \right) \frac{ z^j}{j!} \right)&=  (q)_\infty \sum_{j=0}^\infty R_{k,j}(q) \frac{z^j}{j!} + \theta_{1,2k-1}(\tau) \\ &= 1 + (q)_\infty\sum_{j=1}^\infty R_{k,j}(q) \frac{z^j}{j!},
\end{align*} 
 if we also use (5.2). From here on, we can now argue the way Bringmann, Pandey and van Ittersum did in their proof of Theorem 1.5 in \cite{bringmann2025mockeisensteinseriesassociated} and obtain the desired result. 
 \end{proof} 

We can even give some more information on the form of some of the Fourier coefficient of the $f_{k,j}$: 
\begin{lemma} \textit{For $k\geq3$ and $j \geq 2$ we have that}
\begin{align*}
    f_{k,j}(\tau)&= - \frac{B_j}{2j} + q^k  + \big( 2^{j}-1\big) q^{k+1} +  \big(3^j-2^j \big)q^{k+2} +...+ \big(k^j-(k-1)^j\big)q^{2k-1}\\&\hspace{0.75cm}+O\big(q^{2k}\big).
\end{align*}
\end{lemma}
\begin{proof} For this we return to the $R_{k,j}(q).$ We have $R_{k,j}(q)=0$ if $j $ is odd. For even $j,$ by (5.3), we have
\begin{align*}
    (q)_\infty R_{k,j}(q)&= 2q^k + 2\big(2^j-1\big)q^{k+1} +2\big(3^j-2^j\big)q^{k+2}+...+2\big(k^j-(k-1)^j\big)q^{2k-1}
    \\ &\hspace{0.75cm}+ O\big(q^{2k}\big).
\end{align*}
So multiplying at least 2 of  these $(q)_\infty R_{k,j_i}(q)$ with even $j_i$ gives a term in the order of $O(q^{2k}).$ So by (5.9), we get
\begin{align*}
         2 \sum_{j=1}^\infty \left(f_{k,j}(\tau) + \frac{B_j}{2j} \right) z^j &=  2\sum_{\substack{j = 2\\j \equiv 0 \text{ (mod 2})}}^\infty \bigg[q^k+ \big(2^j-1\big)q^{k+1}  +\big(3^j-2^j\big)q^{k+2} + ...\\ &\hspace{2.5cm} +\big(k^j-(k-1)^j\big)q^{2k-1}  +O(q^{2k})\bigg] z^j. \qedhere
\end{align*}
\end{proof} 
We close by giving the first few Fourier coefficients of some of the $f_{k,j}$ for $k=3,4,5$ and $j$ ranging from $2$ to $8:$
\flushleft{\textbf{Examples:}} \textit{Using Theorem \ref{thm:rec} we get, for $k=3,$
\begin{align*}
    f_{3,2}(\tau)&= -\frac{1}{24} +q^3+3q^4+5q^5+ 7q^6 +9q^7+11q^8+O\left(q^9\right),\\
    f_{3,4}(\tau)&=\frac{1}{240} +q^3 + 15q^4 + 65q^5 +169 q^6 + 333q^7 + 557q^8 + O\big(q^9\big),\\
    f_{3,6}(\tau)&=- \frac{1}{504} +q^3 +63q^4+665q^5+3337q^6+10989q^7 +27581q^8 +O\big(q^9\big),\\
    f_{3,8}(\tau)&=\frac{1}{480} +q^3 + 255 q^4 + 6305q^5 + 58849q^6 + 319293q^7 +1216037 q^8 +O\big(q^9\big).
\end{align*}
With $k=4$ we have
\begin{align*}
    f_{4,2}(\tau)&= -\frac{1}{24} +q^4+3q^5+5q^6+ 7q^7 +9q^8+11q^9+O\big(q^{10}\big),\\
    f_{4,4}(\tau)&=\frac{1}{240} +q^4 + 15q^5 + 65q^6 +175 q^7 + 363q^8 + 635q^9 + O\big(q^{10}\big),\\
    f_{4,6}(\tau)&=- \frac{1}{504} +q^4 +63q^5+665q^6+3367q^7+11499q^7 +30491q^9 +O\big(q^{10}\big),\\
    f_{4,8}(\tau)&=\frac{1}{480} +q^4 + 255 q^5 + 6305q^6 + 58975q^7 + 324963q^8 +1283195 q^9 +O\big(q^{10}\big),
\end{align*}
and lastly, with $k=5$
\begin{align*}
    f_{5,2}(\tau)&= -\frac{1}{24} +q^5+3q^6+5q^7+ 7q^8 +9q^9+11q^{10}+O\big(q^{11}\big),\\
    f_{5,4}(\tau)&=\frac{1}{240} +q^5 + 15q^6 + 65q^7 +175 q^8 + 369q^9 + 665q^{10} + O\big(q^{11}\big),\\
    f_{5,6}(\tau)&=- \frac{1}{504} +q^5 +63q^6+665q^7+3367q^8+11529q^9 +31001q^{10} +O\big(q^{11}\big), \\
    f_{5,8}(\tau)&=\frac{1}{480} +q^5 + 255 q^6 + 6305q^7 + 58975q^8 + 325089q^9 +1288865q^{10} +O\big(q^{11} \big).
\end{align*}}

\end{document}